\documentclass{amsart}

\usepackage{amsmath, amsthm, amssymb} 
\usepackage{graphicx} 
\usepackage{caption}
\usepackage{color} 
\usepackage{hyperref} 
\usepackage{mathrsfs} 

\usepackage[headings]{fullpage}

\newtheorem{thm}{Theorem}[section] %
\newtheorem{lemma}[thm]{Lemma} %
\newtheorem{cor}[thm]{Corollary} %
\newtheorem{prop}[thm]{Proposition} %
\theoremstyle{definition} %
\newtheorem{defn}[thm]{Definition} %
\newtheorem{remark}[thm]{Remark} %
\newtheorem{example}[thm]{Example} %
\newtheorem*{claim*}{Claim} %
\newtheorem*{ack*}{Acknowledgements} %
\newtheorem*{ex*}{Example} %

\DeclareMathOperator{\Out}{Out} %
\DeclareMathOperator{\CV}{CV} %
\DeclareMathOperator{\Aut}{Aut} %
\DeclareMathOperator{\MCG}{MCG} %
\DeclareMathOperator{\UPG}{UPG} %
\DeclareMathOperator{\Stab}{Stab} %
\DeclareMathOperator{\rk}{rank} %
\DeclareMathOperator{\tw}{tw} %

\title{Distortion for Abelian Subgroups of $\Out(F_n)$}%
\author{Derrick Wigglesworth} %
\address{\tt Department of Mathematics, University of Utah, 155 S.\
  1400 E.\, Salt Lake City, UT 84112 \newline
  http://www.math.utah.edu/\~{}dwiggles/} %
\email{\tt dwiggles@math.utah.edu}

\thanks{\today}

\begin{document}

\begin{abstract} We prove that abelian subgroups of the outer
  automorphism group of a free group are quasi-isometrically embedded.
  Our proof uses recent developments in the theory of train track maps
  by Feighn-Handel.  As an application, we prove the rank conjecture
  for $\text{Out}(F_n)$.
\end{abstract}

\thanks{ The author is partially supported by the NSF grant of Mladen
  Bestvina (DMS-1607236) and also acknowledges support from
  U.S. National Science Foundation grants DMS-1246989 and
  DMS-0932078.}


\maketitle

\section{Introduction}
\label{sec:intro}
Given a finitely generated group $G$, a finitely generated subgroup
$H$ is \emph{undistorted} if the inclusion $H\hookrightarrow G$ is a
quasi-isometric embedding with respect to the word metrics on $G$ and
$H$ for some (any) finite generating sets.  A standard technique for
showing that a subgroup is undistorted involves finding a space on
which $G$ acts nicely and constructing a height function on this space
satisfying certain properties: elements which are large in the word
metric on $H$ should change the height function by a lot, elements of
a fixed generating set for $G$ should change the function by a
uniformly bounded amount.  In this paper, we use a couple of
variations of this method.

Let $R_n$ be the wedge of $n$ circles and let $F_n$ be its fundamental
group, the free group of rank $n\geq 2$.  The outer automorphism group
of the free group, $\Out(F_n)$, is defined as the quotient of
$\Aut(F_n)$ by the inner automorphisms, those which arise from
conjugation by a fixed element.  Much of the study of $\Out(F_n)$
draws parallels with the study of mapping class groups.  Furthermore,
many theorems concerning $\Out(F_n)$ and their proofs are inspired by
analogous theorems and proofs in the context of mapping class groups.
Both groups satisfy the Tits alternative \cite{Mc-Tits,BFHI}, both
have finite virtual cohomological dimension \cite{Har-VCD,CV-VCD}, and
both have Serre's property FA to name a few.  Importantly, this
approach to the study of $\Out(F_n)$ has yielded a classification of
its elements in analogy with the Nielsen-Thurston classification of
elements of the mapping class group \cite{BH-TT}, along with
constructive ways for finding good representatives of these elements
\cite{FH-CTS}.

In \cite{FLM-Rank1}, the authors proved that infinite cyclic subgroups
of the mapping class group are undistorted.  Their proof also implies
that higher rank abelian subgroups are undistorted.  In \cite{Ali-TL},
Alibegovi\'{c} proved that infinite cyclic subgroups of $\Out(F_n)$
are undistorted.  In contrast with the mapping class group setting,
her proof does not directly apply to higher rank subgroups: the
question of whether all abelian subgroups of $\Out(F_n)$ are
undistorted has been left open.  In this paper, we answer this in the
affirmative.

\newtheorem*{thm:undistorted}{Theorem \ref{thm:undistorted}}
\begin{thm:undistorted}
  Abelian subgroups of $\Out(F_n)$ are undistorted.
\end{thm:undistorted}

This theorem has implications for various open problems in the study
of $\Out(F_n)$.  In \cite{BM-Rank}, Behrstock and Minsky prove that
the geometric rank of the mapping class group is equal to the maximal
rank of an abelian subgroup of the mapping class group.  As an
application of Theorem \ref{thm:undistorted}, we prove the analogous
result in the $\Out(F_n)$ setting.

\newtheorem*{cor:rank-conj}{Corollary \ref{cor:rank}}
\begin{cor:rank-conj}
  The geometric rank of $\Out(F_n)$ is $2n-3$, which is the maximal
  rank of an abelian subgroup of $\Out(F_n)$.
\end{cor:rank-conj}

We remark that in principle, this could have been done earlier by
using the techniques in \cite{Ali-TL} to show that a specific maximal
rank abelian subgroup is undistorted.

In the course of proving Theorem \ref{thm:undistorted}, we show that,
up to finite index, only finitely many marked graphs are needed to get
good representatives of every element of an abelian subgroup of
$\Out(F_n)$.  In the setting of mapping class groups, the analogous
statement is that for a surface $S$ and an abelian subgroup $H$ of
$\MCG(S)$ there is a Thurston decomposition of $S$ into disjoint
subsurfaces which is respected by every element of $H$.  This can also
be viewed as a version of the Kolchin Theorem of \cite{BFHII} for
abelian subgroups.  We prove:

\newtheorem*{finite-graphs}{Proposition \ref{prop:finite-graphs}}
\begin{finite-graphs}
  For any abelian subgroup $H$ of $\Out(F_n)$, there exists a finite
  index subgroup $H'$ such that every $\phi\in H'$ can be realized as
  a CT on one of finitely many marked graphs.
\end{finite-graphs}

The paper is outlined as follows:

In section \ref{sec:trans} we prove that the translation distance of
an arbitrary element $\phi$ of $\Out(F_n)$ acting on Outer Space is
the maximum of the logarithm of the expansion factors associated to
the exponentially growing strata in a relative train track map for
$\phi$.  This result was obtained previously and independently by
Richard Wade in his thesis \cite{Wade-Thesis}.  This is the analog for
$\Out(F_n)$ of Bers' result \cite{Bers-Trans} that the translation
distance of a mapping class $f$ acting on Teichm\"{u}ller space
endowed with the Teichm\"{u}ller metric is the maximum of the
logarithms of the dilatation constants for the pseudo-Anosov
components in the Thurston decomposition of $f$.  In section
\ref{sec:EG-case} we then use our result on translation distance to
prove the main theorem in the special case where the abelian subgroup
$H$ has ``enough'' exponential data.  More precisely, we will prove
the result under the assumption that the collection of expansion
factor homomorphisms determines an injective map $H\to \mathbb{Z}^N$.

In section \ref{sec:finite-graphs} we prove Proposition
\ref{prop:finite-graphs} and then use this in section
\ref{sec:PG-case} to prove the main result in the case that $H$ has
``enough'' polynomial data.  This is the most technical part of the
paper because we need to obtain significantly more control over the
types of sub-paths that can occur in nice circuits in a marked graph
than was previously available.  The bulk of the work goes towards
proving Proposition \ref{prop:split-circuit}.  This result provides a
connection between the comparison homomorphisms introduced in
\cite{FH-AB} (which are only defined on subgroups of $\Out(F_n)$) and
Alibegovi\'{c}'s twisting function.  We then use this connection to
complete the proof of our main result in the polynomial case.

Finally, in section \ref{sec:mixed-case} we consolidate results from
previous sections to prove Theorem \ref{thm:undistorted}.  The methods
used in sections \ref{sec:EG-case} and \ref{sec:PG-case} can be
carried out with minimal modification in the general setting.

I would like to thank my advisor Mladen Bestvina for many hours of his
time and for his patience.  I would also like to thank Mark Feighn for
his encouragement and support.  Finally, I would also like to express
my gratitude to Radhika Gupta for patiently listening to me go on
about completely split paths for weeks on end and to MSRI for its
hospitality and partial support.

\section{Preliminaries}
\label{sec:preliminaries}
Identify $F_n$ with $\pi_1(R_n,*)$ once and for all.  A \emph{marked
  graph} $G$ is a finite graph of rank $n$ with no valence one
vertices equipped with a homotopy equivalence $\rho\colon R_n\to G$
called a \emph{marking}.  The marking identifies $F_n$ with
$\pi_1(G)$.  As such, a homotopy equivalence $f\colon G\to G$
determines an (outer) automorphism $\phi$ of $F_n$.  We say that
$f\colon G\to G$ \emph{represents} $\phi$.  All homotopy equivalences
will be assumed to map vertices to vertices and the restriction to any
edge will be assumed to be an immersion.

Let $\Gamma$ be the universal cover of the marked graph $G$.  A
\emph{path} in $G$ (resp.\ $\Gamma$) is either an isometric immersion
of a (possibly infinite) closed interval $\sigma\colon I\to G$ (resp.\
$\Gamma$) or a constant map $\sigma\colon I\to G$ (resp.\ $\Gamma$).
If $\sigma$ is a constant map, the path will be called \emph{trivial}.
If $I$ is finite, then any map $\sigma\colon I\to G$ (resp.\ $\Gamma$)
is homotopic rel endpoints to a unique path $[\sigma]$.  We say that
$[\sigma]$ is obtained by \emph{tightening} $\sigma$.  If
$f\colon G\to G$ is a homotopy equivalence and $\sigma$ is a path in
$G$, we define $f_\#(\sigma)$ as $[f(\sigma)]$.  If
$\tilde{f}\colon \Gamma\to \Gamma$ is a lift of $f$, we define
$\tilde{f}_\#$ similarly.  If the domain of $\sigma$ is finite, then
the image has a natural decomposition into edges $E_1E_2\cdots E_k$
called the \emph{edge path associated to} $\sigma$.

A \emph{circuit} is an immersion $\sigma\colon S^1\to G$.  For any
path or circuit, let $\overline{\sigma}$ be $\sigma$ with its
orientation reversed.  A decomposition of a path or circuit into
subpaths is a \emph{splitting} for $f\colon G\to G$ and is denoted
$\sigma=\ldots\sigma_1\cdot\sigma_2\ldots$ if
$f^k_\#(\sigma)=\ldots f^k_\#(\sigma_1)f^k_\#(\sigma_2)\ldots$ for all
$k\geq 1$.

Let $G$ be a graph.  An unordered pair of oriented edges $\{E_1,E_2\}$
is a \emph{turn} if $E_1$ and $E_2$ have the same initial endpoint.
As with paths, we denote by $\overline{E}$, the edge $E$ with the
opposite orientation.  If $\sigma$ is a path which contains
$\ldots\overline{E}_1 E_2\ldots$ or $\ldots\overline{E}_1 E_2\ldots$
in its edge path, then we say $\sigma$ \emph{takes} the turn
$\{E_1,E_2\}$.  A \emph{train track structure} on $G$ is an
equivalence relation on the set of edges of $G$ such that
$E_1\sim E_2$ implies $E_1$ and $E_2$ have the same initial vertex.  A
turn $\{E_1,E_2\}$ is \emph{legal} with respect to a train track
structure if $E_1\nsim E_2$.  A path is legal if every turn crossed by
the associated edge path is legal.  The equivalence classes of this
relation are called \emph{gates}.  A homotopy equivalence
$f\colon G\to G$ induces a train track structure on $G$ as follows.
$f$ determines a map $Df$ on oriented edges in $G$ by definining
$Df(E)$ to be the first edge in the edge path $f(E)$.  We then declare
$E_1\sim E_2$ if $D(f^k)(E_1)=D(f^k)(E_2)$ for some $k\geq 1$.

A \emph{filtration} for a representative $f\colon G\to G$ of an outer
automorphism $\phi$ is an increasing sequence of $f$-invariant
subgraphs $\emptyset=G_0\subset G_1\subset\cdots\subset G_m=G$.  We
let $H_i=\overline{G_i\setminus G_{i-1}}$ and call $H_i$ the
\emph{$i$-th stratum}.  A turn with one edge in $H_i$ and the other in
$G_{i-1}$ is called \emph{mixed} while a turn with both edges in $H_i$
is called a \emph{turn in $H_i$}.  If $\sigma\subset G_i$ does not
contain any illegal turns in $H_i$, then we say $\sigma$ is
\emph{$i$-legal}.

We denote by $M_i$ the submatrix of the transition matrix for $f$
obtained by deleting all rows and columns except those labeled by
edges in $H_i$.  For the representatives that will be of interest to
us, the transition matrices $M_i$ will come in three flavors: $M_i$
may be a zero matrix, it may be the $1\times 1$ identity matrix, or it
may be an irreducible matrix with Perron-Frobenius eigenvalue
$\lambda_i>1$.  We will call $H_i$ a \emph{zero} (Z),
\emph{non-exponentially growing} (NEG), or \emph{exponentially
  growing} (EG) stratum according to these possibilities.  Any stratum
which is not a zero stratum is called an \emph{irreducible stratum}.

\begin{defn}[{\cite{BH-TT}}]
  We say that $f\colon G\to G$ is a \emph{relative train track map}
  representing $\phi\in\Out(F_n)$ if for every exponentially growing
  stratum $H_r$, the following hold:
  \begin{description}
  \item[(RTT-i)]\label{RTT-i} $Df$ maps the set of oriented edges in
    $H_r$ to itself; in particular all mixed turns are legal.
  \item[(RTT-ii)]\label{RTT-ii} If $\sigma\subset G_{r-1}$ is a
    nontrivial path with endpoints in $H_r\cap G_{r-1}$, then so is
    $f_\#(\sigma)$.
  \item[(RTT-iii)]\label{RTT-iii} If $\sigma\subset G_r$ is $r$-legal,
    then $f_\#(\sigma)$ is $r$-legal.
  \end{description}
\end{defn}

Suppose that $u<r$, that $H_u$ is irreducible, $H_r$ is EG and each
component of $G_r$ is non-contractible, and that for each $u<i<r$,
$H_i$ is a zero stratum which is a component of $G_{r-1}$ and each
vertex of $H_i$ has valence at least two in $G_r$.  Then we say that
$H_i$ is \emph{enveloped by $H_r$} and we define
$H_r^z=\bigcup_{k=u+1}^rH_k$.

A path or circuit $\sigma$ in a representative $f\colon G\to G$ is
called a \emph{periodic Nielsen path} if $f_\#^k(\sigma)=\sigma$ for
some $k\geq 1$.  If $k=1$, then $\sigma$ is a \emph{Nielsen path}.  A
Nielsen path is \emph{indivisible} if it cannot be written as a
concatenation of non-trivial Nielsen paths.  If $w$ is a closed
root-free Nielsen path and $E_i$ is an edge such that
$f(E_i)=E_iw^{d_i}$, then we say \emph{$E$ is a linear edge} and we
call $w$ the \emph{axis} of $E$.  If $E_i,E_j$ are distinct linear
edges with the same axis such that $d_i\neq d_j$ and $d_i,d_j>0$, then
we call a path of the form $E_iw^*\overline{E}_j$ an \emph{exceptional
  path}.  In the same scenario, if $d_i$ and $d_j$ have different
signs, we call such a path a \emph{quasi-exceptional path}.  We say
that $x$ and $y$ are \emph{Nielsen equivalent} if there is a Nielsen
path $\sigma$ in $G$ whose endpoints are $x$ and $y$.  We say that a
periodic point $x\in G$ is \emph{principal} if neither of the
following conditions hold:
\begin{itemize}
\item $x$ is not an endpoint of a non-trivial periodic Nielsen path
  and there are exactly two periodic directions at $x$, both of which
  are contained in the same EG stratum.
\item $x$ is contained in a component $C$ of periodic points that is
  topologically a circle and each point in $C$ has exactly two
  periodic directions.
\end{itemize}
A relative train track map $f$ is called \emph{rotationless} if each
principal periodic vertex is fixed and if each periodic direction
based at a principal vertex is fixed.  We remark that there is a
closely related notion of an outer automorphism $\phi$ being
rotationless.  We will not need this definition, but will need the
following relevant facts from \cite{FH-AB}:

\begin{thm}[{\cite[Corollary 3.5]{FH-AB}}]
  There exists $k>0$ depending only on $n$, so that $\phi^k$ is
  rotationless for every $\phi\in\Out(F_n)$.
\end{thm}
\begin{thm}[{\cite[Corollary 3.14]{FH-AB}}]
  For each abelian subgroup $A$ of $\Out(F_n)$, the set of
  rotationless elements in $A$ is a subgroup of finite index in $A$.
\end{thm}

For an EG stratum, $H_r$, we call a non-trivial path
$\sigma\subset G_{r-1}$ with endpoints in $H_r\cap G_{r-1}$ a
\emph{connecting path for $H_r$}.  Let $E$ be an edge in an
irreducible stratum, $H_r$ and let $\sigma$ be a maximal subpath of
$f_\#^k(E)$ in a zero stratum for some $k\geq 1$.  Then we say that
$\sigma$ is \emph{taken}.  A non-trivial path or circuit $\sigma$ is
called \emph{completely split} if it has a splitting
$\sigma=\tau_1\cdot \tau_2\cdots\tau_k$ where each of the $\tau_i$'s
is a single edge in an irreducible stratum, an indivisible Nielsen
path, an exceptional path, or a connecting path in a zero stratum
which is both maximal and taken.  We say that a relative train track
map is \emph{completely split} if $f(E)$ is completely split for every
edge $E$ in an irreducible stratum \emph{and} if for every taken
connecting path $\sigma$ in a zero stratum, $f_\#(\sigma)$ is
completely split.

\begin{defn}[{\cite{FH-Recog}}] \label{def:ct} A relative train track
  map $f\colon G\to G$ and filtration $\mathcal{F}$ given by
  $\emptyset=G_0\subset G_1\subset\cdots\subset G_m=G$ is said to be a
  CT if it satisfies the following properties.
  \begin{description}
  \item[(Rotationless)] $f\colon G\to G$ is rotationless.
  \item[(Completely Split)] $f\colon G\to G$ is completely split.
  \item[(Filtration)] $\mathcal{F}$ is reduced.  The core of each
    filtration element is a filtration element.
  \item[(Vertices)] The endpoints of all indivisible periodic
    (necessarily fixed) Nielsen paths are (necessarily principal)
    vertices.  The terminal endpoint of each non-fixed NEG edge is
    principal (and hence fixed).
  \item[(Periodic Edges)] Each periodic edge is fixed and each
    endpoint of a fixed edge is principal.  If the unique edge $E_r$
    in a fixed stratum $H_r$ is not a loop then $G_{r-1}$ is a core
    graph and both ends of $E_r$ are contained in $G_{r-1}$.
  \item[(Zero Strata)] If $H_i$ is a zero stratum, then $H_i$ is
    enveloped by an EG stratum $H_r$, each edge in $H_i$ is $r$-taken
    and each vertex in $H_i$ is contained in $H_r$ and has link
    contained in $H_i \cup H_r$.
  \item[(Linear Edges)] For each linear $E_i$ there is a closed
    root-free Nielsen path $w_i$ such that $f(E_i) = E_i w_i^{d_i}$
    for some $d_i \ne 0$.  If $E_i$ and $E_j$ are distinct linear
    edges with the same axes then $w_i = w_j$ and $d_i \ne d_j$.
  \item[(NEG Nielsen Paths)] If the highest edges in an indivisible
    Nielsen path $\sigma$ belong to an NEG stratum then there is a
    linear edge $E_i$ with $w_i$ as in (Linear Edges) and there exists
    $k \ne 0$ such that $\sigma = E_i w_i^k \bar E_i$.
  \item[(EG Nielsen Paths)] If $H_r$ is EG and $\rho$ is an
    indivisible Nielsen path of height $r$, then
    $f|G_r = \theta\circ f_{r-1}\circ f_{r}$ where :
    \begin{enumerate}
    \item $f_r : G_r \to G^1$ is a composition of proper extended
      folds defined by iteratively folding $\rho$.
    \item $f_{r-1} : G^1 \to G^2$ is a composition of folds involving
      edges in $G_{r-1}$.
    \item $\theta : G^2 \to G_r$ is a homeomorphism.
    \end{enumerate}
  \end{description}
\end{defn}

We remark that several of the properties in Definition \ref{def:ct}
use terms that have not been defined.  We will not use these
properties in the sequel.  The main result for CTs is the following
existence theorem:

\begin{thm}[{\cite[Theorem 4.28]{FH-Recog}}]
  Every rotationless $\phi\in \Out(F_n)$ is represented by a CT
  $f\colon G\to G$.
\end{thm}

For completely split paths and circuits, all cancellation under
iteration of $f_{\#}$ is confined to the individual terms of the
splitting.  Moreover, $f_\#(\sigma)$ has a complete splitting which
refines that of $\sigma$.  Finally, just as with improved relative
train track maps introduced in \cite{BFHI}, every circuit or path with
endpoints at vertices eventually is completely split.

Culler and Vogtmann's outer space, $\CV_n$, is defined as the space of
homothety classes of free minimal actions of $F_n$ on simplicial
metric trees.  Outer Space has a (non-symmetric) metric defined in
analogy with the Teichm\"{u}ller metric on Teichm\"{u}ller space.  The
distance from $T$ to $T'$ is defined as the logarithm of the infimal
Lipschitz constant among all $F_n$-equivariant maps $f\colon T\to T'$.

Let $\Gamma$ be the universal cover of the marked graph $G$.  Each
non-trivial $c\in F_n$ acts by a \emph{covering translation}
$T_c\colon\Gamma\to \Gamma$ which is a hyperbolic isometry, and
therefore has an \emph{axis} which we denote by $A_c$.  The projection
of $A_c$ to $G$ is the circuit corresponding to the conjugacy class
$c$.  If $E$ is a linear edge in a CT so that $f(E)=Ew^d$ as in
(Linear Edges), then we say $w$ is the axis of $E$.

The \emph{space of lines} in $\Gamma$ is denoted
$\tilde{\mathcal{B}}(\Gamma)$ and is the set
$((\partial \Gamma\times\partial \Gamma)\setminus
\Delta)/\mathbb{Z}_2$ (where $\Delta$ denotes the diagonal and
$\mathbb{Z}_2$ acts by interchanging the factors) equipped with the
compact-open topology.  The space of abstract lines is denoted by
$\partial^2F_n$ and defined by
$((\partial F_n\times \partial F_n)\setminus \Delta)/\mathbb{Z}_2$.
The action of $F_n$ on $\partial F_n$ (resp.\ $\partial \Gamma$)
induces an action on $\partial^2F_n$ (resp.\
$\tilde{\mathcal{B}}(\Gamma)$).  The marking of $G$ defines an
$F_n$-equivariant homeomorphism between $\partial^2F_n$ and
$\tilde{\mathcal{B}}(\Gamma)$.  The quotient of
$\tilde{\mathcal{B}}(\Gamma)$ by the $F_n$ action is the \emph{space
  of lines} in $G$ and is denoted $\mathcal{B}(G)$.  The space of
abstract lines in $R_n$ is denoted by $\mathcal{B}$.

A \emph{lamination}, $\Lambda$, is a closed set of lines in $G$, or
equivalently, a closed $F_n$-invariant subset of
$\tilde{\mathcal{B}}(\Gamma)$.  The elements of a lamination are its
\emph{leaves}.  Associated to each $\phi\in\Out(F_n)$ is a finite
$\phi$-invariant set of \emph{attracting laminations}, denoted by
$\mathcal{L}(\phi)$.  In the coordinates given by a relative train
track map $f\colon G\to G$ representing $\phi$, the attracting
laminations for $\phi$ are in bijection with the EG strata of $G$.

For each attracting lamination $\Lambda^+\in\mathcal{L}(\phi)$, there
is an associated \emph{expansion factor homomorphism},
$PF_{\Lambda^+}\colon \Stab_{\Out(F_n)}(\Lambda^+)\to \mathbb{Z}$
which has been studied in \cite{BFHI}.  We briefly describe the
essential features of $PF_{\Lambda^+}$ here, but we direct the reader
to \cite{BFHI} for more details on lines, laminations, and expansion
factor homomorphisms.  For each $\psi\in\Stab(\Lambda^+)$, at most one
of $\mathcal{L}(\psi)$ and $\mathcal{L}(\psi^{-1})$ can contain
$\Lambda^+$.  If neither $\mathcal{L}(\psi)$ nor
$\mathcal{L}(\psi^{-1})$ contains $\Lambda^+$, then
$PF_{\Lambda^+}(\psi)=0$.  Let $f\colon G\to G$ be a relative train
track map representing $\psi$.  If $\Lambda^+\in\mathcal{L}(\psi)$ and
$H_r$ is the EG stratum of $G$ associated to $\Lambda^+$ with
corresponding PF eigenvalue $\lambda_r$, then
$PF_{\Lambda^+}(\psi)=\log \lambda_r$.  Conversely, if
$\Lambda^+\in\mathcal{L}(\psi^{-1})$, then
$PF_{\Lambda^+}(\psi)=-\log\lambda_r$, where $\lambda_r$ is the PF
eigenvalue for the EG stratum of a RTT representative of $\psi^{-1}$
which is associated to $\Lambda^+$.  The image of $PF_{\Lambda^+}$ is
a discrete subset of $\mathbb{R}$ which we will frequently identify
with $\mathbb{Z}$.

For $\phi\in \Out(F_n)$, each element $\Lambda^+\in\mathcal{L}(\phi)$
has a \emph{paired lamination} in $\mathcal{L}(\phi^{-1})$ which is
denoted by $\Lambda^-$.  The paired lamination is characterized by the
fact that it has the same free factor support as $\Lambda^+$.  That
is, the minimal free factor carrying $\Lambda^+$ is the same as that
which carries $\Lambda^-$.  We denote the pair
$\{\Lambda^+,\Lambda^-\}$ by $\Lambda^\pm$.

\section{Translation Lengths in $\CV_n$}
\label{sec:trans}
In this section, we will compute the translation distance for an
arbitrary element of $\Out(F_n)$ acting on Outer Space.  As is
standard, for $\phi\in\Out(F_n)$ we define the translation distance of
$\phi$ on Outer Space as
$\tau(\phi)=\lim_{n\to\infty}\frac{d(x,x\cdot \phi^n)}{n}$.  It is
straightforward to check that this is independent of $x\in\CV_n$.  For
the remainder of this section $\phi\in\Out(F_n)$ will be fixed, and
$f\colon G\to G$ will be a relative train track map representing
$\phi$ with filtration
$\emptyset=G_0\subset G_1\subset\cdots\subset G_m=G$.

\begin{lemma}
  \label{lem:taulengthfn}
  If $H_r$ is an exponentially growing stratum of $G$, then there
  exists a metric $\ell$ on $G$ such that
  $\ell(f_\#(E))\geq \lambda_r \ell(E)$ for every edge $E\in H_r$,
  where $\lambda_r$ is the Perron-Frobenius eigenvalue associated to
  $H_r$.
\end{lemma}
\begin{proof}
  Let $M_r$ be the transition matrix for the exponentially growing
  stratum, $H_r$ and let $\mathbf{v}$ be a left eigenvector for the PF
  eigenvalue $\lambda_r$ with components $(\mathbf{v})_i$.  Normalize
  $\mathbf{v}$ so that $\sum (\mathbf{v})_i=1$.  For $E_i\in H_r$
  define $\ell(E_i)=(\mathbf{v})_i$.  If $E\notin H_r$ define
  $\ell(E)=1$.  We now check the condition on the growth of edges in
  the EG stratum $H_r$.

  If $E$ is an edge in $H_r$, (RTT-iii) implies that $f(E)$ is
  $r$-legal.  Now write $f_\#(E)=f(E)$ as an edge path,
  $f_\#(E)=E_1E_2\ldots E_j$, and we have
  \begin{equation*}
    \ell(f(E))=\ell(f_\#(E))
    =\sum_{i=1}^j\ell(E_i)
    \geq \sum_{i=1}^j\ell(E_i\cap H_r)
    =\lambda_r\ell(E)
  \end{equation*}
  completing the proof of the lemma.
\end{proof}

We define the $r$-length $\ell_r$ of a path or circuit in $G$ by
ignoring the edges in other strata.  Explicitly,
$\ell_r(\sigma)=\ell(\sigma\cap H_r)$, where $\sigma\cap H_r$ is
considered as a disjoint union of sub-paths of $\sigma$.  Note that
the definition of $\ell$ and the proof of the previous lemma show that
$\ell_r(f_\#(E_i))=\lambda_r\ell(E_i)$.

\begin{lemma}
  \label{rlengthgrows}
  If $\sigma$ is an $r$-legal reduced edge path in $G$ and $\ell$ is
  the metric defined in Lemma \ref{lem:taulengthfn}, then
  $\ell_r(f_\#\sigma)=\lambda_r\ell_r(\sigma)$.
\end{lemma}

\begin{proof}
  We write $\sigma=a_1b_1a_2\cdots b_j$ as a decomposition into
  maximal subpaths where $a_j\subset H_r$ and $b_j\subset G_{r-1}$ as
  in Lemma 5.8 of \cite{BH-TT}.  Applying the lemma, we conclude that
  $f_\#(\sigma)=f(a_1)\cdot f_\#(b_1)\cdot f(a_2)\cdot\ldots\cdot
  f_\#(b_j)$.  Thus,
  \begin{equation*}
    \ell_r(f_\#\sigma)
    =\sum_i\ell_r(f(a_i))+\sum_i\ell_r(f_\#(b_i))
    =\sum_i\ell_r(f(a_i))=\sum_i\lambda_r\ell_r(a_i)
    =\lambda_r\ell_r(\sigma)  \qedhere
  \end{equation*}
\end{proof}

\begin{thm}[\cite{Wade-Thesis}]
  \label{taulength}
  Let $\phi\in \Out(F_n)$ with $f\colon G\to G$ a RTT representative.
  For each EG stratum $H_r$ of $f$, let $\lambda_r$ be the associated
  PF eigenvalue. Then
  $\tau(\phi)=\max\{0,\log\lambda_r\mid H_r\text{ is an EG
    stratum}\}$.
\end{thm}

\begin{proof}
  We first show that $\tau(\phi)\geq \log \lambda_r $ for every EG
  stratum $H_r$.  Let $x=(G,\ell,\text{id})$ where $\ell$ is the
  length function provided by Lemma \ref{lem:taulengthfn}.  Recall
  \cite{FM-Metric} that the logarithm of the factor by which a
  candidate loop is stretched gives a lower bound on the distance
  between two points in $\CV_n$.  Let $\sigma$ be an $r$-legal circuit
  contained in $G_r$ of height $r$ and let
  $C=\ell_r(\sigma)/\ell(\sigma)$. (RTT-iii) implies that
  $f^n_\#(\sigma)$ is $r$-legal for all $n$, so repeatedly applying
  Lemma \ref{rlengthgrows}, we have
  \begin{equation*}
    \frac{\ell(f^n_\#\sigma)}{\ell(\sigma)}
    \geq\frac{\ell_r(f^n_\#\sigma)}{\ell(\sigma)}
    =\frac{\ell_r(f_\#^n\sigma)}{\ell_r(f_\#^{n-1}\sigma)}\frac{\ell_r(f_\#^{n-1}\sigma)}{\ell_r(f_\#^{n-2}\sigma)}\cdots
    \frac{\ell_r(f_\#\sigma)}{\ell_r(\sigma)}\frac{\ell_r(\sigma)}{\ell(\sigma)}
    \geq\lambda_r^n C
  \end{equation*}
  Rearranging the inequality, taking logarithms and using the result
  of \cite{FM-Metric} yields
  \begin{equation*}
    \frac{d(x,x\cdot\phi^n)}{n}\geq\frac{\log(\lambda_r^n C)}{n}=\log\lambda_r+\frac{\log C}{n}
  \end{equation*}
  Taking the limit as $n\to\infty$, we have a lower bound on the
  translation distance of $\phi$.

  For the reverse inequality, fix $\epsilon>0$.  We must find a point
  in outer space which is moved by no more than
  $\epsilon+\max\{0,\log\lambda_r\}$.  The idea is to choose a point
  in the simplex of $\CV_n$ corresponding to a relative train track
  map for $\phi$ in which each stratum is much larger than the
  previous one.  This way, the metric will see the growth in every EG
  stratum.  Let $f\colon G\to G$ be a relative train track map as
  before, but assume that each NEG stratum consists of a single edge.
  This is justified, for example by choosing $f$ to be a CT
  \cite{FH-Recog}.  Let $K$ be the maximum edge length of the image of
  any edge of $G$.  Define a length function on $G$ as follows:
  \begin{equation*}
    \ell(E)=
    \begin{cases}
      (K/\epsilon)^r & \text{if }E\text{ is the unique edge in the NEG stratum }H_r\\
      (K/\epsilon)^r & \text{if }E\text{ is an edge in the zero stratum }H_r\\
      (K/\epsilon)^r\cdot v_i & \text{if }E_i\in H_r\text{ and
      }H_r\text{ is an EG stratum with }\vec{v}\text{ as above}
    \end{cases}
  \end{equation*}
  The logarithm of the maximum amount that any edge is stretched in a
  difference of markings map gives an upper bound on the Lipschitz
  distance between any two points.  So we just check the factor by
  which every edge is stretched.  Clearly the stretch factor for edges
  in fixed strata is 1.  If $E$ is the single edge in an NEG stratum,
  $H_i$, then
  \begin{equation*}
    \frac{\ell(f(E))}{\ell(E)}
    \leq \frac{\ell(E)+K\max\{\ell(E')\mid E'\in G_{i-1}\}}{\ell(E)}
    =\frac{(K/\epsilon)^i+K(K/\epsilon)^{i-1}}{(K/\epsilon)^i}
    =1+\epsilon
  \end{equation*}
  Similarly, if $E$ is an edge in the zero stratum, $H_i$, then
  \begin{equation*}
    \frac{\ell(f(E))}{\ell(E)}
    \leq\frac{K(K/\epsilon)^{i-1}}{(K/\epsilon)^i}
    = \epsilon
  \end{equation*}
  We will use the notation $\ell_r^{\downarrow}(\sigma)$ to denote the
  length of the intersection of $\sigma$ with $G_{r-1}$.  So for any
  path $\sigma$ contained in $G_r$, we have
  $\ell(\sigma)=\ell_r(\sigma)+\ell_r^{\downarrow}(\sigma)$.  Now, if
  $E_i$ is an edge in the EG stratum, $H_r$, with normalized PF
  eigenvector $\mathbf{v}$ then
  \begin{equation*}
    \frac{\ell(f(E_i))}{\ell(E_i)}
    =\frac{\ell_r(f(E_i))+\ell_r^{\downarrow}(f(E_i))}{\ell(E_i)}
    =\lambda_r+\frac{\ell_r^{\downarrow}(f(E_i))}{\ell(E_i)} \leq
    \lambda_r+\frac{K(K/\epsilon)^{r-1}}{(K/\epsilon)^r (\mathbf{v})_i}
    =\lambda_r+\frac{\epsilon}{(\mathbf{v})_i}
  \end{equation*}
  Since the vector $\mathbf{v}$ is determined by $f$, after replacing
  $\epsilon$ we have that
  $\frac{\ell(f(E))}{\ell(E)}\leq \max\{\lambda_r,1\}+\epsilon$ for
  every edge of $G$.  Thus, the distance $(G,\ell,\rho)$ is moved by
  $\phi$ is less than $\max\{\log(\lambda_r),0\}+\epsilon$ and the
  proof is complete.
\end{proof}

Now that we have computed the translation distance of an arbitrary
$\phi$ acting on outer space, we'll use this result to establish our
main result in a special case.

\section{The Exponential Case}
\label{sec:EG-case}

In this section, we'll analyze the case that the abelian subgroup
$H=\langle \phi_1,\ldots,\phi_k\rangle$ has enough exponential data so
that the entire group is seen by the so called lambda map.  More
precisely, given an attracting lamination $\Lambda^+$ for an outer
automorphism $\phi$, let
$PF_{\Lambda^+}\colon\Stab(\Lambda^+)\to \mathbb{Z}$ be the expansion
factor homomorphism defined by Corollary 3.3.1 of \cite{BFHI}.  In
\cite[Corollary 3.14]{FH-AB}, the authors prove that every abelian
subgroup of $\Out(F_n)$ has a finite index subgroup which is
rotationless (meaning that every element of the subgroup is
rotationless).  Distortion is unaffected by passing to a finite index
subgroup, so there is no loss in assuming that $H$ is rotationless.
Now let $\mathcal{L}(H)=\bigcup_{\phi\in H}\mathcal{L}(\phi)$ be the
set of attracting laminations for elements of $H$.  By \cite[Lemma
4.4]{FH-AB}, $\mathcal{L}(H)$ is a finite set of $H$-invariant
laminations.  Define $PF_H\colon H\to\mathbb{Z}^{\#\mathcal{L}(H)}$ by
taking the collection of expansion factor homomorphisms for attracting
laminations of the subgroup $H$.  In what follows, we will need to
interchange $PF_{\Lambda^+}$ for $PF_{\Lambda^-}$ and for that we will
need the following lemma.

\begin{lemma}\label{lem:paired-lamination}
  If $\Lambda^+\in\mathcal{L}(\phi)$ and
  $\Lambda^-\in\mathcal{L}(\phi^{-1})$ are paired laminations then
  $\frac{PF_{\Lambda^+}}{PF_{\Lambda^-}}$ is a constant map.  That is,
  $PF_{\Lambda^+}$ and $PF_{\Lambda^-}$ differ by a multiplicative
  constant, and so determine the same homomorphism.
\end{lemma}

\begin{proof}
  First, Corollary 1.3(2) of \cite{HM-FSII} gives that
  $\Stab(\Lambda^+)=\Stab(\Lambda^-)$ (which we will henceforth refer
  to as $\Stab(\Lambda^\pm)$), so the ratio in the statement is always
  well defined.  Now $PF_{\Lambda^+}$ and $PF_{\Lambda^-}$ each
  determine a homomorphism from $\Stab(\Lambda^{\pm})$ to $\mathbb{R}$
  and it suffices to show that these homomorphisms have the same
  kernel.  Suppose $\psi\notin\ker PF_{\Lambda^+}$ so that by
  \cite[Corollary 3.3.1]{BFHI} either $\Lambda^+\in \mathcal{L}(\psi)$
  or $\Lambda^+\in\mathcal{L}(\psi^{-1})$.  After replacing $\psi$ by
  $\psi^{-1}$ if necessary, we may assume
  $\Lambda^+\in\mathcal{L}(\psi)$.  Now $\psi$ has a paired lamination
  $\Lambda_\psi^-\in\mathcal{L}(\psi^{-1})$ which a priori could be
  different from $\Lambda^-$.  But Corollary 1.3(1) of \cite{HM-FSII}
  says that in fact $\Lambda_\psi^-=\Lambda^-$ and therefore that
  $\Lambda^-\in\mathcal{L}(\psi^{-1})$.  A final application of
  \cite[Corollary 3.3.1]{BFHI} gives that
  $\psi\notin\ker PF_{\Lambda^-}$.  This concludes the proof.
\end{proof}

\begin{thm}
  \label{thm:EG-case}
  If $PF_H$ is injective, then $H$ is undistorted in $\Out(F_n)$.
\end{thm}
\begin{proof}
  Let $k$ be the rank of $H$ and start by choosing laminations
  $\Lambda_1,\ldots,\Lambda_k\in\mathcal{L}(H)$ so the restriction of
  the function $PF_H$ to the coordinates determined by
  $\Lambda_1,\ldots,\Lambda_k$ is still injective.  First note that
  $\{\Lambda_1,\ldots,\Lambda_k\}$ cannot contain an
  attracting-repelling lamination pair by Lemma
  \ref{lem:paired-lamination}.

  Next, pass to a finite index subgroup of $H$ and choose generators
  $\phi_i$ so that after reordering the $\Lambda_i$'s if necessary,
  each generator satisfies
  $PF_H(\phi_i)=(0,\ldots,0,PF_{\Lambda_i}(\phi_i),0,\ldots,0)$.  Let
  $*\in\CV_n$ be arbitrary and let
  $\psi=\phi_1^{p_1}\cdots\phi_k^{p_k}\in H$.  We complete the proof
  one orthant at a time by replacing some of the $\phi_i$'s by their
  inverses so that all the $p_i$'s are non-negative.  Next, after
  replacing some of the $\Lambda_i$'s by their paired laminations
  (again using Lemma \ref{lem:paired-lamination}), we may assume that
  $PF_H(\psi)$ has all coordinates nonnegative.

  By Theorem \ref{taulength}, the translation distance of $\psi$ is the
  maximum of the Perron-Frobenius eigenvalues associated to the EG
  strata of a relative train track representative $f$ of $\psi$.
  Some, but not necessarily all, of $\Lambda_1,\ldots,\Lambda_k$ are
  attracting laminations for $\psi$.  Those $\Lambda_i$'s which are in
  $\mathcal{L}(\psi)$ are associated to EG strata of $f$.  For such a
  stratum, the logarithm of the PF eigenvalue is
  $PF_{\Lambda_i}(\psi)$ and the fact that $PF_{\Lambda_i}$ is a
  homomorphism implies
  \begin{equation*}
    PF_{\Lambda_i}(\psi)
    =PF_{\Lambda_i}(\phi_1^{p_1}\cdots\phi_k^{p_k})
    =p_1 PF_{\Lambda_j}(\phi_1)+\ldots+p_k PF_{\Lambda_j}(\phi_k)
    =p_i PF_{\Lambda_i}(\phi_i)
  \end{equation*}

  Thus, the translation distance of $\psi$ acting on outer space is
  \begin{align*}
    \tau(\psi)
    &=\max\{\log\lambda\mid\lambda\text{ is PF eigenvalue associated to an EG stratum of }\psi\}\\
    &\geq\max\{PF_{\Lambda_i}(\psi)\mid \Lambda_i\text{ is in }\mathcal{L}(\psi)\text{ and }1\leq i\leq k\}\\
    &=\max\{p_iPF_{\Lambda_i}(\phi_i)\mid 1\leq i\leq k\}
  \end{align*}

  In the last equality, the maximum is taken over a larger set, but
  the only values added to the set were 0.

  Let $S$ be a symmetric (i.e., $S^{-1}=S$) generating set for
  $\Out(F_n)$ and let $D_1=\max_{s\in S}d(*,*\cdot s)$.  If we write
  $\psi$ in terms of the generators $\psi=s_1s_2\cdots s_l$, then
  \begin{align*}
    d(*,*\cdot\psi)
    &\leq d(*,*\cdot s_l)+d(*\cdot s_l,*\cdot s_{l-1}s_l)+\ldots+d(*\cdot(s_2\ldots s_l),*\cdot (s_1\ldots s_l))\\
    &= d(*,*\cdot s_l)+d(*,*\cdot s_{l-1}+\ldots+d(*,*\cdot s_1)
      \leq D_1|\psi|_{\Out(F_n)}
  \end{align*}

  Let
  $K_1=\min\{PF_{\Lambda_i^\pm}(\phi_j^\pm)\mid 1\leq i,j\leq k\}$.
  Rearranging this and combining these inequalities, we have
  \begin{equation*}
    |\psi|_{\Out(F_n)}
    \geq\frac{1}{D_1}d(*,*\cdot \psi)
    \geq\frac{1}{D_1}\tau(\psi)
    \geq\frac{1}{D_1}\max\{p_iPF_{\Lambda_i}(\phi_i)\mid 1\leq i\leq k\}
    \geq\frac{K_1}{D_1}\max\{p_i\}
  \end{equation*}
  We have thus proved that the image of $H$ under the injective
  homomorphism $PF_H$ is undistorted in $\mathbb{Z}^k$.  To conclude
  the proof, recall that any injective homomorphism between abelian
  groups is a quasi-isometric embedding.
\end{proof}

Now that we have established our result in the exponential setting, we
move on to the polynomial case.  First we prove a general result about
CTs representing elements of abelian subgroups.

\section{Abelian Subgroups are Virtually Finitely Filtered}
\label{sec:finite-graphs}
In this section, we prove an analog of \cite[Theorem 1.1]{BFHII} for
abelian subgroups.  In that paper, the authors prove that any
unipotent subgroup of $\Out(F_n)$ is contained in the subgroup
$\mathcal{Q}$ of homotopy equivalences respecting a fixed filtration
on a fixed graph $G$.  They call such a subgroup ``filtered''.  While
generic abelian subgroups of $\Out(F_n)$ are \emph{not} unipotent, we
prove that they are virtually filtered.  Namely, that such a subgroup
is virtually contained in the union of finitely many $\mathcal{Q}$'s.
First, we review the comparison homomorphisms introduced in
\cite{FH-AB}.

\subsection{Comparison Homomorphisms}
\label{sec:comp-homos}
Feighn and Handel defined certain homomorphisms to $\mathbb{Z}$ which
measure the growth of linear edges and quasi-exceptional families in a
CT representative.  Though they can be given a canonical description
in terms of principal lifts, we will only need their properties in
coordinates given by a CT.  Presently, we will define these
homomorphisms and recall some basic facts about them.  Complete
details on comparison homomorphisms can be found in \cite{FH-AB}.

Comparison homomorphisms are defined in terms of principal sets for
the subgroup $H$.  The exact definition of a principal set is not
important for us.  We only need to know that a \emph{principal set}
$\mathcal{X}$ for an abelian subgroup $H$ is a subset of
$\partial F_n$ which defines a lift $s\colon H\to\Aut(F_n)$ of $H$ to
the automorphism group.  Let $\mathcal{X}_1$ and $\mathcal{X}_2$ be
two principal sets for $H$ that define distinct lifts $s_1$ and $s_2$
to $\Aut(F_n)$.  Suppose further that $\mathcal{X}_1\cap\mathcal{X}_2$
contains the endpoints of an axis $A_c$.  Since $H$ is abelian,
$s_1\cdot s_2^{-1}\colon H\to \Aut(F_n)$ defined by
$s_1\cdot s_2^{-1}(\phi)=s_1(\phi)\cdot s_2(\phi)^{-1}$ is a
homomorphism.  It follows from \cite[Lemma 4.14]{FH-Recog} that for
any $\phi\in H$, $s_1(\phi)=s_2(\phi)i_c^k$ for some $k$, where
$i_c\colon \Aut(F_n)\to \Aut(F_n)$ denotes conjugation by $c$.
Therefore $s_1\cdot s_2^{-1}$ defines homomorphism into
$\langle i_c\rangle$, which we call the \emph{comparison homomorphism}
determined by $\mathcal{X}_1$ and $\mathcal{X}_2$.  Generally, we will
use the letter $\omega$ for comparison homomorphisms.

For a rotationless abelian subgroup $H$, there are only finitely many
comparison homomorphisms \cite[Lemma 4.3]{FH-AB}.  Let $K$ be the
number of distinct comparison homomorphisms and (as before) let $N$ be
the number of attracting laminations for $H$.  The map
$\Omega\colon H\to \mathbb{Z}^{N+K}$ defined as the product of the
comparison homomorphisms and expansion factor homomorphisms is
injective \cite[Lemma 4.6]{FH-AB}.  An element $\phi\in H$ is called
\emph{generic} if every coordinate of
$\Omega(\phi)\in\mathbb{Z}^{N+K}$ is non-zero.  If $\phi$ is generic
and $f\colon G\to G$ is a CT representing $\phi$, then there is a
correspondence between the comparison homomorphisms for $H$ and the
linear edges and quasi-exceptional families in $G$ described in the
introduction to \S 7 of \cite{FH-AB} which we briefly describe now.
There is a comparison homomorphism $\omega_{E_i}$ for each linear edge
$E_i$ in $G$. If $f(E_i)=E_i\cdot u^{d_i}$, then
$\omega_{E_i}(\phi)=d_i$.  There is also a comparison homomorphism for
each quasi-exceptional family, $E_iu^*\overline{E}_j$ which is denoted
by $\omega_{E_iu^*\overline{E}_j}$.  If $E_i$ is as before and
$f(E_j)=E_ju^{d_j}$, then $\omega_{E_iu^*\overline{E}_j}$ and
$\omega(\phi)=d_i-d_j$.  We illustrate this correspondence with an
example.

\begin{example}
  Let $G=R_3$ be the rose with three petals labeled $a,b,$ and $c$.
  For $i,j\in\mathbb{Z}$, define $g_{i,j}\colon G\to G$ as follows:
  \begin{equation*}
    \begin{aligned}[c]
      a&\mapsto a\\
      g_{i,j}\colon\,b&\mapsto ba^i\\
      c&\mapsto ca^j
    \end{aligned}
  \end{equation*}
  Each $g_{i,j}$ determines an outer automorphism of $F_3$ which we
  denote by $\phi_{i,j}$.  The automorphisms $\phi_{i,j}$ all lie in
  the rank two abelian subgroup
  $H=\langle \phi_{0,1},\phi_{1,0}\rangle$.  The subgroup $H$ has
  three comparison homomorphisms which are easily understood in the
  coordinates of a CT for a generic element of $H$.  The element
  $\phi_{2,1}$ is generic in $H$, and $g_{2,1}$ is a CT representing
  it.  Two of the comparison homomorphisms manifest as $\omega_b$ and
  $\omega_c$ where $\omega_b(\phi_{i,j})=i$ and
  $\omega_c(\phi_{i,j})=j$.  The third homomorphism is denoted by
  $\omega_{ba^*\overline{c}}$ and it measures how a path of the form
  $ba^*\overline{c}$ changes when $g_{i,j}$ is applied.  Since
  $g_{i,j}(ba^*\overline{c})=ba^{*+i-j}\overline{c}$, we have
  $\omega_{ba^*\overline{c}}(\phi_{i,j})=i-j$.

\end{example}

In the sequel, we will rely heavily on this correspondence between the
comparison homomorphisms of $H$ and the linear edges and
quasi-exceptional families in a CT for a generic element of $H$.  We
now prove the main result of this section.

\begin{prop}
  \label{prop:finite-graphs}
  For any abelian subgroup $H$ of $\Out(F_n)$, there exists a finite
  index subgroup $H'$ such that every $\phi\in H'$ can be realized as
  a CT on one of finitely many marked graphs.
\end{prop}

Most of the proof consists of restating and combining results of
Feighn and Handel from \cite{FH-AB}.  We refer the reader to \S 6 of
their paper for the relevant notation and most of the relevant
results.
\begin{proof}
  First replace $H$ by a finite index rotationless subgroup
  \cite[Corollary 3.14]{FH-AB}.  The proof is by induction on the rank
  of $H$.  The base case follows directly from \cite[Lemma
  6.18]{FH-AB}.  Let $H=\langle \phi\rangle$ and let
  $f^\pm\colon G^\pm\to G^\pm$ be CT's for $\phi$ and $\phi^{-1}$
  which are both generic in $H$.  The definitions then guarantee that
  $\mathbf{i}=(i,i,\ldots,i)$ for $i>0$ is both generic and
  admissible.  Lemma 6.18 then says that
  $f^\pm_{\mathbf{i}}\colon G^\pm\to G^\pm$ is a CT representing
  $\phi^\pm_{\mathbf{i}}=\phi^{\pm i}$, so we are done.

  Assume now that the claim holds for all abelian subgroups of rank
  less than $k$, and let $H=\langle\phi_1,\ldots,\phi_k\rangle$.  The
  set of generic elements of $H$ is the complement of a finite
  \cite[Lemma 4.3]{FH-AB} collection of hyperplanes.  Every
  non-generic element, $\phi$, lies in a rank $(k-1)$ abelian subgroup
  of $H$: the kernel of the corresponding comparison homomorphism.  By
  induction and the fact that there are only finitely many
  hyperplanes, every non-generic element has a CT representative on
  one of finitely many marked graphs.  We now add a single marked
  graph for each sector defined by the complement of the hyperplanes.

  Let $\phi$ be generic and let $f\colon G\to G$ be a CT
  representative.  Let $\mathcal{D}(\phi)$ be the disintegration of
  $\phi$ as defined in \cite{FH-AB} and recall that
  $\mathcal{D}(\phi)\cap H$ is finite index in $H$ \cite[Theorem
  7.2]{FH-AB}.  Let $\Gamma$ be the semigroup of generic elements of
  $\mathcal{D}(\phi)\cap H$ that lie in the same sector of $H$ as
  $\phi$ (i.e., for every $\gamma\in\Gamma$ and every coordinate
  $\omega$ of $\Omega$, the signs of $\omega(\gamma)$ and
  $\omega(\phi)$ agree).  The claim is that every element of $\Gamma$
  can be realized as a CT on the marked graph $G$ and we will show
  this by explicitly reconstructing the generic tuple $\mathbf{a}$
  such that $\gamma=[f_{\mathbf{a}}]$.  Fix $\gamma\in \Gamma$ and let
  $\phi_{\mathbf{a}_1},\ldots,\phi_{\mathbf{a}_k}$ be a generating set
  for $H$ with $\mathbf{a}_i$ generic \cite[Corollary 6.20]{FH-AB}.
  Write $\gamma$ as a word in the generators,
  $\gamma=\phi_{\mathbf{a}_1}^{j_1}\cdots\phi_{\mathbf{a}_k}^{j_k}$
  and define $\mathbf{a}=j_1\mathbf{a}_1+\ldots+j_k\mathbf{a}_k$.
  Since the admissibility condition is a set of homogeneous linear
  equations which must be preserved under taking linear combinations,
  as long as every coordinate of $\mathbf{a}$ is non-negative,
  $\mathbf{a}$ must be admissible.  To see that every coordinate of
  $\mathbf{a}$ is in fact positive, let $\omega$ be a coordinate of
  $\Omega^\phi$.  Using the fact that $\omega$ is a homomorphism to
  $\mathbb{Z}$ and repeatedly applying \cite[Lemma 7.5]{FH-AB} to the
  $\phi_{\mathbf{a}_i}$'s, we have
  \begin{align*}
    \omega(\gamma)
    &=j_1\omega(\phi_{\mathbf{a}_1})+j_2\omega(\phi_{\mathbf{a}_2})+\ldots+j_k\omega(\phi_{\mathbf{a}_k})\\
    &=j_1(\mathbf{a}_1)_s\omega(\phi)+j_2(\mathbf{a}_2)_s\omega(\phi)+\ldots+j_k(\mathbf{a}_k)_s\omega(\phi)\\
    &=(j_1\mathbf{a}_1+j_2\mathbf{a}_2+\ldots+j_k\mathbf{a}_k)_s\omega(\phi)\\
    &=(\mathbf{a})_s\omega(\phi)
  \end{align*}
  where $(\mathbf{a})_s$ denotes the $s$-th coordinate of the vector
  $\mathbf{a}$.  Since $\gamma$ and $\phi$ were assumed to be generic
  and to lie in the same sector, we conclude that every coordinate of
  $\mathbf{a}$ is positive.  The injectivity $\Omega^\phi$ \cite[Lemma
  7.4]{FH-AB} then implies that $\gamma=[f_{\mathbf{a}}]$.  That
  $\mathbf{a}$ is in fact generic follows from the fact, which is
  directly implied by the definitions, that if $\mathbf{a}$ is a
  generic tuple, then $\phi_{\mathbf{a}}$ is a generic element of $H$.
  Finally, we apply \cite[Lemma 6.18]{FH-AB} to conclude that
  $f_{\mathbf{a}}\colon G\to G$ is a CT.  Thus, every element of
  $\Gamma$ has a CT representative on the marked, filtered graph $G$.
  Repeating this argument in each of the finitely many sectors and
  passing to the intersection of all the finite index subgroups
  obtained this way yields a finite index subgroup $H'$ and finitely
  many marked graphs, so that every generic element of $H'$ can be
  realized as a CT on one of the marked graphs.  The non-generic
  elements were already dealt with using the inductive hypothesis, so
  the proof is complete.
\end{proof}

\section{The Polynomial Case}
\label{sec:PG-case}
In \cite{Ali-TL}, the author introduced a function that measures the
twisting of conjugacy classes about an axis in $F_n$ and used this
function to prove that cyclic subgroups of $\UPG$ are undistorted.  In
order to use the comparison homomorphisms in conjunction with this
twisting function, we need to establish a result about the possible
terms occuring in completely split circuits.  After establishing this
connection, we use it to prove (Theorem \ref{thm:PG-case}) the main
result under the assumption that $H$ has ``enough'' polynomial data.

In the last section, we saw the correspondence between comparison
homomorphisms and certain types of paths in a CT.  In order to use the
twisting function from \cite{Ali-TL}, our goal is to find circuits in
$G$ with single linear edges or quasi-exceptional families as
subpaths, and moreover to do so in such a way that we can control
cancellation at the ends of these subpaths under iteration of $f$.
This is the most technical section of the paper, and the one that most
heavily relies on the use of CTs.  The main result is Proposition
\ref{prop:split-circuit}.

\subsection{Completely Split Circuits}
\label{sec:split-circuits}
One of the main features of train track maps is that they allow one to
understand how cancellation occurs when tightening $f^k(\sigma)$ to
$f^k_{\#}(\sigma)$.  In previous incarnations of train track maps,
this cancellation was understood inductively based on the height of
the path $\sigma$.  One of the main advantages of completely split
train track maps is that the way cancellation can occur is now
understood directly, rather than inductively.

Given a CT $f\colon G\to G$ representing $\phi$, the set of allowed
terms in completely split paths would be finite were it not for the
following two situations: a linear edge $E\mapsto E u$ gives rise to
an infinite family of INPs of the form $Eu^*\overline{E}$, and two
linear edges with the same axis
$E_1\mapsto E_1u^{d_1},\,E_2\mapsto E_2u^{d_2}$ (with $d_1$ and $d_2$
having the same sign) give rise to an infinite family of exceptional
paths of the form $E_1 u^*\overline{E}_2$.  To see that these are the
only two subtleties, one only needs to know that there is at most one
INP of height $r$ for each EG stratum $H_r$.  This is precisely
\cite[Corollary 4.19]{FH-AB}.

To connect Feighn-Handel's comparison homomorphisms to
Alibegovi\'{c}'s twisting function, we would like to show that every
linear edge and exceptional family occurs as a term in the complete
splitting of some completely split circuit.  We will in fact show
something stronger:

\begin{prop}
  \label{prop:split-circuit}
  There is a completely split circuit $\sigma$ containing every
  allowable term in its complete splitting.  That is the complete
  splitting of $\sigma$ contains at least one instance of every
  \begin{itemize}
  \item edge in an irreducible stratum (fixed, NEG, or EG)
  \item maximal, taken connecting subpath in a zero stratum
  \item infinite family of INPs $Eu^*\overline{E}$
  \item infinite family of exceptional paths $E_1u^*\overline{E}_2$
  \end{itemize}
\end{prop}

The proof of this proposition will require a careful study of
completely split paths.  With that aim, we define a directed graph
that encodes the complete splittings of such paths.  Given a CT
$f\colon G\to G$ representing $\phi$ define a di-graph
$\mathcal{CSP}(f)$ (or just $\mathcal{CSP}$ when $f$ is clear) whose
vertices are oriented allowed terms in completely split paths.  More
precisely, there are two vertices for each edge in an irreducible
stratum: one labeled by $E$ and one labeled by $\overline{E}$ (which
we will refer to at $\tau_E$ and $\tau_{\overline{E}}$).  There are
two vertices for each maximal taken connecting path in a zero stratum:
one for $\sigma$ and one for $\overline{\sigma}$ (which will be
referred to as $\tau_{\sigma}$ and $\tau_{\overline{\sigma}}$).
Similarly, there are two vertices for each family of exceptional
paths, two vertices for each INP of EG height, and \emph{one} vertex
for each infinite family of NEG Nielsen paths.  There is only one
vertex for each family of indivisible Nielsen path $\sigma$ whose
height is NEG because $\sigma$ and $\overline{\sigma}$ determine the
same initial direction.  There is an edge connecting two vertices
$\tau_\sigma$ and $\tau_{\sigma'}$ in $\mathcal{CSP}(f)$ if the path
$\sigma\sigma'$ is completely split with splitting given by
$\sigma\cdot \sigma'$.  This is equivalent to the turn
$(\overline{\sigma},\sigma')$ being legal by the uniqueness of
complete splittings \cite[Lemma 4.11]{FH-Recog}.

Any completely split path (resp.\ circuit) $\sigma$ with endpoints at
vertices in $G$ defines a directed edge path (resp.\ directed loop) in
$\mathcal{CSP}(f)$ given by reading off the terms in the complete
splitting of $\sigma$.  Conversely, a directed path or loop in
$\mathcal{CSP}(f)$ yields a not quite well defined path or circuit
$\sigma$ in $G$ which is necessarily completely split.  The only
ambiguity lies in how to define $\sigma$ when the path in
$\mathcal{CSP}(f)$ passes through a vertex labeled by a Nielsen path
of NEG height or a quasi-exceptional family.

\begin{example}
  \label{ex:iwip-rose}
  Consider the rose $R_2$ consisting of two edges $a$ and $b$ with the
  identity marking.  Let $f\colon R_2\to R_2$ be defined by
  $a\mapsto ab,\, b\mapsto bab$.  This is a CT representing a fully
  irreducible outer automorphism.  There is one indivisible Nielsen
  path $\sigma=ab\overline{a}\overline{b}$.  The graph
  $\mathcal{CSP}(f)$ is shown in Figure \ref{fig:CSPRose}.  The blue
  edges represent the fact that each of the paths
  $\overline{b}\cdot\overline{b},\, \overline{b}\cdot\overline{a},\,
  \overline{b}\cdot\sigma$, and $\overline{b}\cdot a$ is completely
  split.

\begin{figure}[h]
  \centering{ \def\svgwidth{.5\linewidth}
    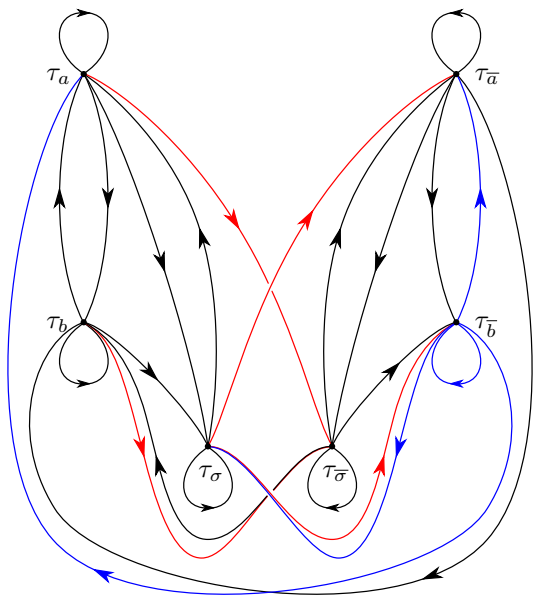
    \caption{The graph of $\mathcal{CSP}(f)$ for Example
      \ref{ex:iwip-rose}}
    \label{fig:CSPRose}
  }
\end{figure}
\end{example}

\begin{remark}
  \label{rmk:edges}
  A basic observation about the graph $\mathcal{CSP}$ is that every
  vertex $\tau_\sigma$ has at least one incoming and at least one
  outgoing edge.  While this is really just a consequence of the fact
  that every vertex in a CT has at least two gates, a bit of care is
  needed to justify this formally.  Indeed, let $v$ be the initial
  endpoint of $\sigma$.  If there is some legal turn $(E,\sigma)$ at
  $v$ where $E$ is an edge in an irreducible stratum, then
  $\overline{E}\cdot \sigma$ is completely split so there is an edge
  in $\mathcal{CSP}$ from $\tau_{\overline{E}}$ to $\tau_\sigma$.  The
  other possibility is that the only legal turns
  $(\underline{\quad},\sigma)$ at $v$ consist of an edge in a zero
  stratum $H_i$.  In this case, (Zero Strata) guarantees that $v$ is
  contained in the EG stratum $H_r$ which envelops $H_i$ and that the
  link of $v$ is contained in $H_i\cup H_r$.  In particular, there are
  a limited number of possibilities for $\sigma$; $\sigma$ may be a
  taken connecting subpath in $H_i$, an edge in $H_r$, or an EG INP of
  height $r$.  In the first two cases, $\sigma$ is a term in the
  complete splitting of $f^k_\#(E)$ for some edge $E$.  By increasing
  $k$ if necessary, we can guarantee that $\sigma$ is not the first or
  last term in this splitting.  Therefore, there is a directed edge in
  $\mathcal{CSP}$ with terminal endpoint $\tau_\sigma$.  In the case
  that $\sigma$ is an INP, $\sigma$ has a first edge $E_0$ which is
  necessarily of EG height.  We have already established that there is
  a directed edge in $\mathcal{CSP}$ pointed to $\tau_{E_0}$, so we
  just observe that any vertex in $\mathcal{CSP}$ with a directed edge
  ending at $E_0$ will also have a directed edge terminating at
  $\tau_\sigma$.  The same argument shows that there is an edge in
  $\mathcal{CSP}$ emanating from $\tau_\sigma$.
\end{remark}

The statement of Proposition \ref{prop:split-circuit} can now be
rephrased as a statement about the graph $\mathcal{CSP}$.  Namely,
that there is a directed loop in $\mathcal{CSP}$ which passes through
every vertex.

We will need some basic terminology from the study of directed graphs.
We say a di-graph $\Gamma$ is \emph{strongly connected} if every
vertex can be connected to every other vertex in $\Gamma$ by a
directed edge path.  In any di-graph, we may define an equivalence
relation on the vertices by declaring $v\sim w$ if there is a directed
edge path from $v$ to $w$ and vice versa (we are required to allow the
trivial edge path so that $v\sim v$).  of $\Gamma$.  The equivalence
classes of this relation partition the vertices of $\Gamma$ into
\emph{strongly connected components}.

We will prove that $\mathcal{CSP}(f)$ is connected and has one
strongly connected component.  From this, Proposition
\ref{prop:split-circuit} follows directly.  The proof proceeds by
induction on the core filtration of $G$, which is the filtration
obtained from the given one by considering only the filtration
elements which are their own cores.  Because the base case is in fact
more difficult than the inductive step, we state it as a lemma.

\begin{lemma}
  \label{lem:base-case}
  If $f\colon G\to G$ is a CT representing a fully irreducible
  automorphism, then $\mathcal{CSP}(f)$ is connected and strongly
  connected.
\end{lemma}

\begin{proof}
  Under these assumptions, there are two types of vertices in
  $\mathcal{CSP}(f)$: those labeled by edges, and those labeled by
  INPs.  We denote by $\mathcal{CSP}_e$ the subgraph consisting of
  only the vertices which are labeled by edges.  Recall that $\tau_E$
  denotes the vertex in $\mathcal{CSP}$ corresponding to the edge $E$.
  If the leaves of the attracting lamination are non-orientable, then
  we can produce a path in $\mathcal{CSP}_e$ starting at $\tau_E$,
  then passing through every other vertex in $\mathcal{CSP}_e$, and
  finally returning to $\tau_E$ by looking at a long segment of a leaf
  of the attracting lamination.  More precisely, (Completely Split)
  says that $f^k(E)$ is a completely split path for all $k\geq 0$ and
  the fact that $f$ is a train track map says that this complete
  splitting contains no INPs.  Moreover, irreducibility of the
  transition matrix and non-orientability of the lamination implies
  that for sufficiently large $k$ this path not only contains every
  edge in $G$ (with both orientations), but contains the edge $E$
  followed by every other edge in $G$ with both of its orientations,
  and then the edge $E$ again.  Such a path in $G$ exactly shows that
  $\mathcal{CSP}_e$ is connected and strongly connected.

  We isolate the following remark for future reference.

  \begin{remark}
    \label{rmk:INPs}
    If there is an indivisible Nielsen path $\sigma$ in $G$, write its
    edge path $\sigma=E_1E_2\ldots E_k$ (recall that all INPs in a CT
    have endpoints at vertices).  If $\tau_{\sigma'}$ is any vertex in
    $\mathcal{CSP}$ with a directed edge pointing to $\tau_{E_1}$,
    then $\sigma'\cdot \sigma$ is completely split since the turn
    $(\overline{\sigma}',\sigma)$ must be legal.  Hence there is also
    a directed edge in $\mathcal{CSP}$ from $\tau_{\sigma'}$ to
    $\tau_\sigma$. The same argument shows that there is an edge in
    $\mathcal{CSP}$ from $\tau_\sigma$ to some vertex
    $\tau'\neq\tau_\sigma$.
  \end{remark}

  Since $\mathcal{CSP}_e$ is strongly connected, and the remark
  implies that each vertex $\tau_\sigma$ (for $\sigma$ an INP in $G$)
  has directed edges coming from and going back into
  $\mathcal{CSP}_e$, we conclude that $\mathcal{CSP}$ is strongly
  connected in the case that leaves of the attracting lamination are
  non-orientable.

  Now choose an orientation on the attracting lamination $\Lambda$.
  If we imagine an ant following the path in $G$ determined by a leaf
  of $\Lambda$, then at each vertex $v$ we see the ant arrive along
  certain edges and leave along others.  Let $E$ be an edge with
  initial vertex $v$ so that $E$ determines a gate $[E]$ at $v$.  We
  say that $[E]$ is a \emph{departure gate} at $v$ if $E$ occurs in
  some (any) oriented leaf $\lambda$.  Similarly, we say the gate
  $[E]$ is an \emph{arrival gate} at $v$ if the edge $\overline{E}$
  occurs in $\lambda$.  Some gates may be both arrival and departure
  gates.

  Suppose now that there is some vertex $v$ in $G$ that has at least
  two arrival gates and some vertex $w$ that has at least two
  departure gates.  As before, we will produce a path in
  $\mathcal{CSP}_e$ that shows this subgraph has one strongly
  connected component.  Start at any edge in $G$ and follow a leaf
  $\lambda$ of the lamination until you have crossed every edge with
  its forward orientation.  Continue following the leaf until you
  arrive at $v$, say through the gate $[\overline{E}]$.  Since $v$ has
  two arrival gates, there is some edge $E'$ which occurs in $\lambda$
  with the given orientation and whose terminal vertex is $v$
  ($[\overline{E'}]$ is a second arrival gate).  Now turn onto
  $\overline{E'}$.  Since $[\overline{E}]$ and $[\overline{E'}]$ are
  distinct gates, this turn is legal.  Follow $\overline{\lambda}$
  going backwards until you have crossed every edge of $G$ (now in the
  opposite direction).  Finally, continue following
  $\overline{\lambda}$ until you arrive at $w$, where there are now
  two arrival gates because you are going backwards.  Use the second
  arrival gate to turn around a second time, and follow $\lambda$ (now
  in the forwards direction again) until you cross the edge you
  started with.  By construction, this path in $G$ is completely split
  and every term in its complete splitting is a single edge.  The
  associated path in $\mathcal{CSP}_e$ passes through every vertex and
  then returns to the starting vertex, so $\mathcal{CSP}_e$ is
  strongly connected.  In the presence of an INP, Remark
  \ref{rmk:INPs} completes the proof of the lemma under the current
  assumptions.

  We have now reduced to the case where the lamination is orientable
  \emph{and} either every vertex has only one departure gate or every
  vertex has only one arrival gate.  The critical case is the latter
  of the two, and we would like to conclude in this situation that
  there is an INP.  Example \ref{ex:iwip-rose} illustrates this
  scenario.  Some edges are colored red to illustrate the fact that in
  order to turn around and get from the vertices labeled by $a$ and
  $b$ to those labeled by $\overline{a}$ and $\overline{b}$, one must
  use an INP.  The existence of an INP in this situation is provided
  by the following lemma.

  \begin{lemma}\label{lem:find-INP}
    Assume $f\colon G\to G$ is a CT representing a fully irreducible
    rotationless automorphism.  Suppose that the attracting lamination
    is orientable and that every vertex has exactly one arrival gate.
    Then $G$ has an INP, $\sigma$, and the initial edges of $\sigma$
    and $\overline{\sigma}$ are oriented consistently with the
    orientation of the lamination.
  \end{lemma}

  We postpone the proof of this lemma and explain how to conclude our
  argument.  If every vertex has one arrival gate, then we apply the
  lemma to conclude that there must be an INP.  Since INPs have
  exactly one illegal turn, using the previous argument, we can turn
  around once.  Now if we are again in a situation where there is only
  one arrival gate, then we can apply the lemma a second time (this
  time with the orientation of $\Lambda$ reversed) to obtain the
  existence of a second INP, allowing us to turn around a second time.
\end{proof}

We remark that since there is at most one INP in each EG stratum of a
CT, Lemma \ref{lem:find-INP} implies that if the lamination is
orientable, then some vertex of $G$ must have at least 3 gates.

\begin{proof}[Proof of Lemma \ref{lem:find-INP}]
  There is a vertex of $G$ that is fixed by $f$ since \cite[Lemma
  3.19]{FH-Recog} guarantees that every EG stratum contains at least
  one principal vertex and principal vertices are fixed by
  (Rotationless).  Choose such a vertex $v$ and let
  $\tilde{v}\in \Gamma$ be a lift of $v$ to the universal cover
  $\Gamma$ of $G$.  Let $g$ be the unique arrival gate at $\tilde{v}$.
  Lift $f$ to a map $\tilde{f}\colon \Gamma\to\Gamma$ fixing
  $\tilde{v}$.  Let $T$ be the infinite subtree of $\Gamma$ consisting
  of all embedded rays $\gamma\colon [0,\infty)\to \Gamma$ starting at
  $\tilde{v}$ and leaving every vertex through its unique arrival
  gate.  That is $\gamma(0)=\tilde{v}$ and whenever $\gamma(t)$ is a
  vertex, $D\gamma(t)$ should be the unique arrival gate at
  $\gamma(t)$.  Refer to Figure \ref{fig:IllegalTree} for the tree $T$
  for Example \ref{ex:iwip-rose}.

  \begin{figure}[h]
    \centering{ \def\svgwidth{.7\linewidth}
      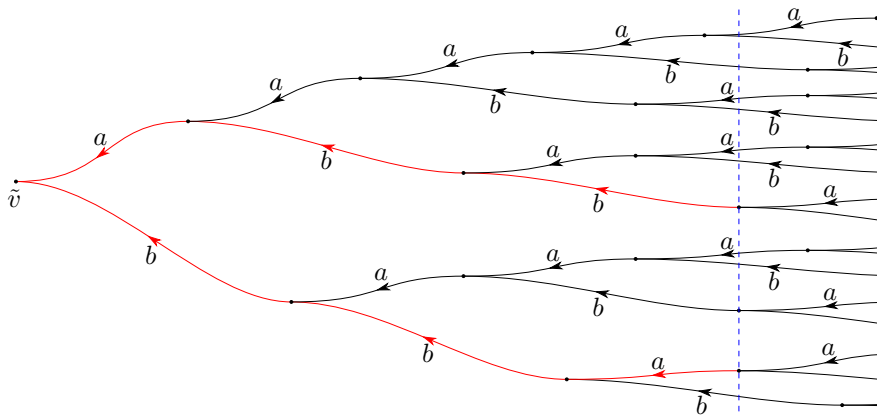
      \caption{The tree $T$ for Example \ref{ex:iwip-rose}.  The red
        path connects two vertices of the same height.}
      \label{fig:IllegalTree}
    }
  \end{figure}

  First, we claim that $\tilde{f}(T)\subset T$.  To see this, notice
  that since $f$ is a topological representative, it suffices to show
  that $\tilde{f}(p)\in T$ for every vertex $p$ of $T$.  Notice that
  vertices $p$ of $T$ are characterized by two things: first
  $[\tilde{v},p]$ is legal, and second, for every edge $E$ in the edge
  path of $[\tilde{v},p]$, the gate $[E]$ is the unique arrival gate
  at the initial endpoint of $E$.  Now
  $[\tilde{v},\tilde{f}(p)]=\tilde{f}([\tilde{v},p])$ is legal because
  $f$ is a train track map.  Moreover, every edge $E$ in the edge path
  of $[\tilde{v},p]$ occurs (with orientation) in a leaf
  $\overline{\lambda}$ of the lamination.  Since $\tilde{f}$ takes
  leaves to leaves preserving orientation, the same is true for
  $\tilde{f}([\tilde{v},p])$.  The gate determined by every edge in
  the edge path of $\overline{\lambda}$ is the unique arrival gate at
  that vertex.  Thus, for every edge $E$ in the edge path of
  $\tilde{f}([\tilde{v},p])$, $[E]$ is the unique arrival gate at that
  vertex, which means that $\tilde{f}(p)\in T$.

  Endow $G$ with a metric using the left PF eigenvector of the
  transition matrix so that for every edge of $G$, we have
  $\ell(f(E))=\nu\,\ell(E)$ where $\nu$ is the PF eigenvalue of the
  transition matrix.  Lift the metric on $G$ to a metric on $\Gamma$
  and define a height function on the tree $T$ by measuring the
  distance to $\tilde{v}$: $h(p)=d(p,\tilde{v})$.  Since legal paths
  are stretched by exactly $\nu$, we have that for any $p\in T$,
  $h(\tilde{f}(p))=\nu h(p)$.

  Now let $w$ and $w'$ be two distinct lifts of $v$ with the same
  height, $h(w)=h(w')$.  To see that this is possible, just take
  $\alpha$ and $\beta$ to be two distinct
  ($\langle\alpha,\beta\rangle\simeq F_2$) circuits in $G$ based at
  $v$ which are obtained by following a leaf of the lamination.  The
  initial vertices of the lifts of $\alpha\beta$ and $\beta\alpha$
  which end at $\tilde{v}$ are distinct lifts of $v$ which are
  contained in $T$, and have the same height.

  Let $\tau$ be the unique embedded segment connecting $w$ to $w'$ in
  $T$.  By \cite[Lemma 4.25]{FH-Recog}, $\tilde{f}^k_\#(\tau)$ is
  completely split for all sufficiently large $k$.  Moreover, the
  endpoints of $\tilde{f}^k_\#(\tau)$ are distinct since the
  restriction of $\tilde{f}$ to the lifts of $v$ is injective.  This
  is simply because
  $\tilde{f}\colon (\Gamma,\tilde{v})\to(\Gamma,\tilde{v})$ represents
  an automorphism of $F_n$ and lifts of $v$ correspond to elements of
  $F_n$.  Now observe that the endpoints $\tilde{f}^k_\#(\tau)$ have
  the same height and for any pair of distinct vertices with the same
  height, the unique embedded segment connecting them must contain an
  illegal turn.  This follows from the definition of $T$ and the
  assumption that every vertex has a unique arrival gate.  Therefore,
  the completely split path $\tilde{f}^k_\#(\tau)$ contains an illegal
  turn.  In particular, it must have an INP in its complete splitting.
  That the initial edges of $\sigma$ and $\overline{\sigma}$ are
  oriented consistently with the orientation on $\lambda$ is evident
  from the construction.
\end{proof}

The key to the inductive step is provided by the ``moving up through
the filtration'' lemma from \cite{FH-AB} which explicitly describes
how the graph $G$ can change when moving from one element of the core
filtration to the next.  Recall the \emph{core filtration} of $G$ is
the filtration
$G_0\subseteq G_{l_1}\subseteq\ldots\subseteq G_{l_k}=G_m=G$ obtained
by restricting to those filtration elements which are their own cores.
For each $G_{l_i}$, the $i$-th stratum of the core filtration is
defined to be $H_{l_i}^c=\bigcup_{j=l_{i-1}+1}^{l_i}H_j$.  Finally, we
let $\Delta \chi_i^-=\chi(G_{l_{i-1}})-\chi(G_{l_i})$ denote the
negative of the change in Euler characteristic.

\begin{lemma}[{\cite[Lemma 8.3]{FH-AB}}]
  \label{lem:moving-up}
  \begin{enumerate}
  \item If $H_{l_i}^c$ does not contain any EG strata then one of the
    following holds.
    \begin{enumerate}
    \item \label{case-1a}$l_i=l_{i-1}+1$ and the unique edge in
      $H_{l_i}^c$ is a fixed loop that is disjoint from $G_{l_{i-1}}$.
    \item \label{case-1b}$l_i=l_{i-1}+1$ and both endpoints of the
      unique edge in $H_{l_i}^c$ are contained in $G_{l_{i-1}}$.
    \item \label{case-1c}$l_i=l_{i-1}+2$ and the two edges in
      $H_{l_i}^c$ are nonfixed and have a common initial endpoint that
      is not in $H_{l_{i-1}}$ and terminal endpoints in $G_{l_{i-1}}$.
    \end{enumerate}
    In case \ref{case-1a}, $\Delta_i\chi^-=0$; in cases \ref{case-1b}
    and \ref{case-1c}, $\Delta_i\chi^-=1$.
  \item \label{case-2} If $H_{l_i}^c$ contains an EG stratum, then
    $H_{l_i}$ is the unique EG stratum in $H_{l_i}^c$ and there exists
    $l_{i-1}\leq u_i<l_i$ such that both of the following hold.
    \begin{enumerate}
    \item For $l_{i_1}< j\le u_i$, $H_j$ is a single nonfixed edge
      $E_j$ whose terminal vertex is in $G_{l_{i-1}}$ and whose
      initial vertex has valence one in $G_{u_i}$. In particular,
      $G_{u_i}$ deformation retracts to $G_{l_{i-1}}$ and
      $\chi(G_{u_i})=\chi(G_{l_{i-1}})$.
    \item For $u_i < j < l_i$, $H_j$ is a zero stratum. In other
      words, the closure of $G_{l_i}\setminus G_{u_i}$ is the extended
      $EG$ stratum $H^z_{l_i}$.
    \end{enumerate}
    If some component of $H^c_{l_i}$ is disjoint from $G_{u_i}$ then
    $H^c_{l_i}=H_{l_i}$ is a component of $G_{l_i}$ and
    $\Delta_i\chi^-\ge 1$; otherwise $\Delta_i\chi^-\ge 2$.
  \end{enumerate}
\end{lemma}

As we move up through the core filtration, we imagine adding new
vertices to $\mathcal{CSP}$ and adding new edges connecting these
vertices to each other and to the vertices already present.  Thus, we
define $\mathcal{CSP}_{l_i}$ to be the subgraph of $\mathcal{CSP}$
consisting of vertices labeled by allowable terms in $G_{l_i}$.  Here
we use the fact that the restriction of $f$ to each connected
component of an element of the core filtration is a CT.

The problem with proving that $\mathcal{CSP}$ is strongly connected by
induction on the core filtration is that $\mathcal{CSP}_{l_i}$ may
have multiple connected components.  This only happens, however, if
$G_{l_i}$ has more than one connected component in which case
$\mathcal{CSP}_{l_i}$ will have multiple connected components.  If any
component of $G_{l_i}$ is a topological circle (necessarily consisting
of a single fixed edge $E$), then $\mathcal{CSP}_{l_i}$ will have two
connected components for this circle.

\begin{lemma}
  For every $1\leq i\leq k$, the number of strongly connected
  components of $\mathcal{CSP}_{l_i}(f)$ is equal to
  \begin{equation*}
    2\cdot\#\left\{\text{connected components of }G_{l_i}\text{that are circles}\right\}
    +\#\left\{\text{connected components of }G_{l_i}\text{that are not circles}\right\}
  \end{equation*}
\end{lemma}

The following proof is in no way difficult.  It only requires a
careful analysis of the many possible cases.  The only case where
there is any real work is in case \ref{case-2} of Lemma
\ref{lem:moving-up}.

\begin{proof}
  Lemma \ref{lem:base-case} establishes the base case when $H_1^c$ is
  exponentially growing.  If $H_1^c$ is a circle, then
  $\mathcal{CSP}_{1}$ has exactly two vertices, each with a self loop,
  so the lemma clearly holds.  We now proceed to the inductive step,
  which is case-by-case analysis based on Lemma \ref{lem:moving-up}.
  We set some notation to be used throughout: $E$ will be an edge with
  initial vertex $v$ and terminal vertex $w$ (it's possible that
  $v=w$).  We denote by $G_{l_i}^v$ the component of $G_{l_i}$
  containing $v$ and similarly for $w$.  Let $\mathcal{CSP}_{l_i}^v$
  be the component(s) of $\mathcal{CSP}_{l_i}$ containing paths which
  pass through $v$.  In the case that $G_{l_i}^v$ is a topological
  circle, there will be two such components.

  In case \ref{case-1a} of Lemma \ref{lem:moving-up},
  $\mathcal{CSP}_{l_i}$ is obtained from $\mathcal{CSP}_{l_{i-1}}$ by
  adding two new vertices: $\tau_E$ and $\tau_{\overline{E}}$.  Each
  new vertex has a self loop, and no other new edges are added.  So
  the number of connected components of $\mathcal{CSP}$ increases by
  two.  Each component is strongly connected.

  In case \ref{case-1b}, there are several subcases according to the
  various possibilities for the edge $E$, and the topological types of
  $G_{l_{i-1}}^v$ and $G_{l_{i-1}}^w$.  First, suppose that $E$ is a
  fixed edge.  Then $\mathcal{CSP}_{l_i}$ is obtained from
  $\mathcal{CSP}_{l_{i-1}}$ by adding two new vertices.  There are no
  new INPs since the restriction of $f$ to each component of $G_{l_i}$
  is a CT and any INP is of the form provided by (NEG Nielsen Paths)
  or (EG Nielsen Paths).  As in Remark \ref{rmk:edges}, the vertex
  $\tau_E$ has an incoming edge with initial endpoint $\tau$ and an
  outgoing edge with terminal endpoint $\tau'$.  Moreover,
  $\tau\in\mathcal{CSP}_{l_{i-1}}^v$ and
  $\tau'\in\mathcal{CSP}_{l_{i-1}}^w$.  We then have a directed edge
  from $\sigma\in\mathcal{CSP}_{l_{i-1}}^w$ to $\tau_{\overline{E}}$
  and a directed edge from $\tau_{\overline{E}}$ to
  $\sigma'\in\mathcal{CSP}_{l_{i-1}}^v$.  Hence, there are directed
  paths in $\mathcal{CSP}_{l_i}$ connecting the two strongly connected
  subgraphs $\mathcal{CSP}_{l_{i-1}}^v$ and
  $\mathcal{CSP}_{l_{i-1}}^w$ to each other, and passing through all
  new vertices.  Therefore, there is one strongly connected component
  of $\mathcal{CSP}_{l_i}$ corresponding to the component of $G_{l_i}$
  containing $v$ (and $w$).  This component cannot be a circle, since
  it contains at least two edges.  In the case that $G_{l_{i-1}}^v$
  (resp.\ $G_{l_{i-1}}^w$) is a topological circle, we remark that
  there are incoming (resp.\ outgoing) edges in
  $\mathcal{CSP}_{l_i}^v$ (resp.\ $\mathcal{CSP}_{l_i}^w$) to $\tau_E$
  from each of the components of $\mathcal{CSP}_{l_{i-1}}^v$ (resp.\
  $\mathcal{CSP}_{l_{i-1}}^w$).  See Figure \ref{fig:NEGConnecting}.

  \begin{figure}[h]
    \centering{ \def\svgwidth{.7\linewidth}
      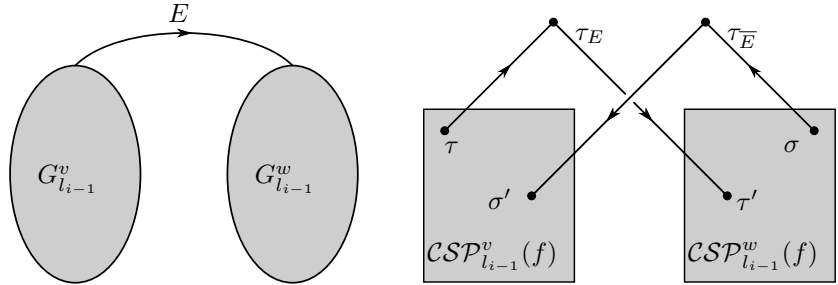
      \caption{A possibility for $G_{l_i}$ and the graph
        $\mathcal{CSP}_{l_i}$ when $H_{l_i}^c$ is a single NEG edge}
      \label{fig:NEGConnecting}
    }
  \end{figure}

  Suppose now that $E$ is a non-fixed NEG edge.  There are two new
  vertices in $\mathcal{CSP}_{l_i}$ labeled $\tau_E$ and
  $\tau_{\overline{E}}$.  The argument given in the previous paragraph
  goes through once we notice that if $v\neq w$, then $G_{l_{i-1}}^w$
  cannot be a circle since this would imply that $w$ is not a
  principal vertex in $G_{l_i}$ (see first bullet point in the
  definition) contradicting the fact that $f|_{G_{l_i}}$ is a CT
  ((Vertices) is not satisfied).

  If $E$ is a non-linear edge, then we are done.  If $E$ is linear,
  then there will be other new vertices in $\mathcal{CSP}_{l_i}$.
  There will be a new vertex for the family of NEG Nielsen paths
  $Eu^*\overline{E}$.  The fact that we have concluded the inductive
  step for the vertex $\tau_E$ along with remark \ref{rmk:INPs} shows
  that this new vertex is in the same strongly connected component as
  $\tau_E$.  There will also be two vertices for each family of
  exceptional paths $Eu^*\overline{E'}$.  For the exact same reasons,
  these vertices are also in this strongly connected component.  This
  concludes the proof in case \ref{case-1b} of Lemma
  \ref{lem:moving-up}.

  The arguments given thus far apply directly to case \ref{case-1c} of
  Lemma \ref{lem:moving-up}.  We remark that in this case, neither of
  the components of $G_{l_{i-1}}$ containing the terminal endpoints of
  the new edges can be circles for the same reason as before.

  The most complicated way that $G$ (and hence $\mathcal{CSP}$) can
  change is when $H_{l_i}^c$ contains an EG stratum.  In case
  \ref{case-2} of Lemma \ref{lem:moving-up}, if some component of
  $H_{l_i}^c$ is disjoint from $G_{u_i}$, then $H_{l_i}^c$ is a
  component of $G_{l_i}$ and the restriction of $f$ to this component
  is a fully irreducible.  In particular, $\mathcal{CSP}_{l_i}$ has
  one more strongly connected component than $\mathcal{CSP}_{l_{i-1}}$
  by Lemma \ref{lem:base-case}.

  Though case \ref{case-2} of Lemma \ref{lem:moving-up} describes
  $G_{l_i}$ as being built from $G_{l_{i-1}}$ in three stages from
  bottom to top, somehow it is easier to prove $\mathcal{CSP}_{l_i}$
  has the correct number of connected components by going from top to
  bottom.

  By looking at a long segment of a leaf of the attracting lamination
  for $H_{l_i}$, we can see as in Lemma \ref{lem:base-case} that the
  vertices in $\mathcal{CSP}_{l_i}$ labeled by edges in the EG stratum
  $H_{l_i}$ are in at most two different strongly connected
  components.  In fact, we can show that these vertices are all in the
  same strongly connected component.  Since we are working under the
  assumption that no component of $H_{l_i}^c$ is disjoint from
  $G_{u_i}$, we can use one of the components of $G_{u_i}$ to turn
  around on a leaf of the lamination.  Indeed, choose some component
  $G^1$ of $G_{u_i}$ which intersects $H_{l_i}$.  Let $E$ be an EG
  edge in $H_{l_i}$ with terminal vertex $w\in G^1$.  Note that if
  $G^1$ deformation retracts onto a circle with vertex $v$, then some
  EG edge in $H_{l_i}$ must be incident to $v$, since otherwise
  $f|_{G_{l_i}}$ would not be a CT.  Thus, by replacing $E$ if
  necessary, we may assume in this situation that $w$ is on the
  circle.  Using the inductive hypothesis and the fact that mixed
  turns are legal, we can connect the vertex $\tau_E$ to the vertex
  $\tau_{\overline{E}}$ in $\mathcal{CSP}_{l_i}$.  Then we can follow
  a leaf of the lamination going backwards until we return to $w$, say
  along $E'$.  If $E=E'$, then the leaves of the lamination were
  non-orientable in the first place, and all the vertices labeled by
  edges in $H_{l_i}$ are in the same strongly connected component of
  $\mathcal{CSP}_{l_i}$.  Otherwise, apply the inductive hypothesis
  again and use the fact that mixed turns are legal to get a path from
  $\tau_{E'}$ to $\tau_{\overline{E}'}$.  This shows all vertices
  labeled by edges in $H_{l_i}$ are in the same strongly connected
  component of $\mathcal{CSP}_{l_i}$.  We will henceforth denote the
  strongly connected component of $\mathcal{CSP}_{l_i}$ which contains
  all these vertices by $\mathcal{CSP}_{l_i}^{EG}$.

  If there is an INP $\sigma$ of height $H_{l_i}$, its first and last
  edges are necessarily in $H_{l_i}$.  Remark \ref{rmk:INPs} then
  implies that $\tau_\sigma$ and $\tau_{\overline{\sigma}}$ are in
  $\mathcal{CSP}_{l_i}^{EG}$.  Recall that the only allowable terms in
  complete splittings which intersect zero strata are connecting paths
  which are both maximal and \emph{taken}.  In particular, each vertex
  in $\mathcal{CSP}_{l_i}$ corresponding to such a connecting path is
  in the aforementioned strongly connected component,
  $\mathcal{CSP}_{l_i}^{EG}$.

  Now let $E$ be an NEG edge in $H_{l_i}^c$ with terminal vertex $w$.
  There is necessarily an outgoing edge from $\tau_{E}$ into
  $\mathcal{CSP}_{l_{i-1}}^w$ and an incoming edge to $\tau_E$ from
  $\mathcal{CSP}_{l_i}^{EG}$.  If the graph $G_{l_{i-1}}^w$ is not a
  topological circle, then the corresponding component
  $\mathcal{CSP}_{l_{i-1}}^w$ is already strongly connected and there
  is a directed edge from this graph back to $\tau_{\overline{E}}$ and
  from there back into $\mathcal{CSP}_{l_i}^{EG}$.  Thus, this
  subgraph is contained in the strongly connected component
  $\mathcal{CSP}_{l_i}^{EG}$.  On the other hand, if $G_{l_{i-1}}^w$
  is a topological circle, then there is a directed edge from
  $\tau_{E}$ back into $\mathcal{CSP}_{l_i}^{EG}$ because mixed turns
  are legal, and as before, some edge in $H_{l_i}$ must be incident to
  $w$.  Thus all the vertices in $\mathcal{CSP}_{l_i}$ labeled by NEG
  edges are in the strongly connected component
  $\mathcal{CSP}_{l_i}^{EG}$, as are all vertices in
  $\mathcal{CSP}_{l_{i-1}}^w$ for $w$ as above.

  The same argument and the inductive hypothesis shows that for any
  component of $G_{l_{i-1}}$ which intersects $H_{l_i}$, the
  corresponding strongly connected component(s) of
  $\mathcal{CSP}_{l_{i-1}}$ are also in $\mathcal{CSP}_{l_i}^{EG}$.
  The only thing remaining is to deal with NEG Nielsen paths and
  families of exceptional paths.  Both of these are handled by Remark
  \ref{rmk:INPs} and the fact that we have already established that
  $\mathcal{CSP}_{l_i}^{EG}$ contains all vertices of the form
  $\tau_E$ or $\tau_{\overline{E}}$ for NEG edges in $H_{l_i}^c$.  We
  have shown that every vertex of a strongly connected component of
  $\mathcal{CSP}_{l_{i-1}}$ coming from a component of $G_{l_{i-1}}$
  which intersects $H_{l_i}^c$ is in the strongly connected component
  $\mathcal{CSP}_{l_i}^{EG}$.  In particular, there is only one
  strongly connected component of $\mathcal{CSP}_{l_i}$ for the
  component of $G_{l_i}$ which contains edges in $H_{l_i}^c$.  This
  completes the proof of the proposition.
\end{proof}

In the proof of Theorem \ref{thm:PG-case}, we will need to consider a
weakening of the complete splitting of paths and circuits.  The
\emph{quasi-exceptional splitting} of a completely split path or
circuit $\sigma$ is the coarsening of the complete splitting obtained
by considering each quasi-exceptional subpath to be a single element.
Given a CT $f\colon G\to G$, we define the graph
$\mathcal{CSP}^{QE}(f)$ by adding two vertices to $\mathcal{CSP}(f)$
for each QE-family (one for $E_iu^*\overline{E}_j$ and one for
$E_ju^*\overline{E}_i$).  For every vertex $\tau_\sigma$ with a
directed edge terminating at $\tau_{E_i}$ add an edge from
$\tau_\sigma$ to $\tau_{E_iu^*\overline{E}_j}$ and similarly for every
edge emanating from $\tau_{\overline{E}_j}$, add an edge to the same
vertex beginning at $\tau_{E_iu^*\overline{E}_j}$.  Do the same for
the vertex $\tau_{E_ju^*\overline{E}_i}$.  As before, every completely
split path $\sigma$ gives rise to a directed edge path in
$\mathcal{CSP}^{QE}$ corresponding to its QE-splitting.  It follows
immediately from the definition and Proposition
\ref{prop:split-circuit} that

\begin{cor}\label{cor:QE-circuit}
  There is a completely split circuit $\sigma$ containing every
  allowable term in its QE-splitting.
\end{cor}

We are now ready to prove our main result in the polynomial case.

\subsection{Polynomial Subgroups are Undistorted}

In this subsection, we will complete the proof of our main result in
the polynomial case.  We first recall the height function defined by
Alibegovi\'{c} in \cite{Ali-TL}.  Given two conjugacy classes
$[u],[w]$ of elements of $F_n$, define the \emph{twisting of $[w]$
  about $[u]$} as
\begin{equation*}
  \tw_u(w)=\max\{k\mid w=au^kb \text{ where }u,w
  \text{ are a cyclically reduced conjugates of }[u],[w]\}
\end{equation*}
Then define the twisting of $[w]$ by
$\tw(w)=\max\{\tw_u(w)\mid u\in F_n\}$.  Alibegovi\'{c} proved the
following lemma using bounded cancellation, which we restate for
convenience.  A critical point is that $D_2$ is independent of $w$.

\begin{lemma}[{\cite[Lemma 2.4]{Ali-TL}}]
  \label{lem:generator-twists}
  There is a constant $D_2$ such that $\tw(s(w))\leq \tw(w)+D_2$ for
  all conjugacy classes $w$ and all $s\in S$, our symmetric finite
  generating set of $\Out(F_n)$.
\end{lemma}

Since we typically work with train tracks, we have a similar notion of
twisting adapted to that setting.  Let $\tau$ be a path or circuit in
a graph $G$ and let $\sigma$ be a circuit in $G$.  Define the
\emph{twisting of $\tau$ about $\sigma$} as
\begin{equation*}
  \tw_\sigma(\tau)=\max\{k\mid \tau=\alpha\sigma^k\beta 
  \text{ where the path }\alpha\sigma^k\beta
  \text{ is immersed} \}
\end{equation*}
Then define $\tw(\tau)=\max\{\tw_\sigma(\tau)\mid \sigma$ is a
circuit$\}$.  The bounded cancellation lemma of \cite{Coo-BCC}
directly implies

\begin{lemma}
  \label{lem:marking-twists}
  If $\rho\colon R_n\to G$ and $[w]$ is a conjugacy class in
  $F_n=\pi_1(R_n)$, then $\tw(\rho(w))\ge \tw(w)-2C_\rho$.
\end{lemma}

We are now ready to prove non-distortion for polynomial abelian
subgroups.  Recall the map $\Omega\colon H\to \mathbb{Z}^{N+K}$ was
defined by taking the product of comparison and expansion factor
homomorphisms.  In the following theorem, we will denote the
restriction of this map to the last $K$ coordinates (those
corresponding to comparison homomorphisms) by $\Omega_{comp}$.

\begin{thm}
  \label{thm:PG-case}
  Let $H$ be a rotationless abelian subgroup of $\Out(F_n)$ and assume
  that the map from $H$ into the collection of comparison factor
  homomorphisms $\Omega_{comp}\colon H\to \mathbb{Z}^K$ is injective.
  Then $H$ is undistorted.
\end{thm}
\begin{proof}
  The first step is to note that it suffices to prove the generic
  elements of $H$ are uniformly undistorted.  This is just because the
  set of non-generic elements of $H$ is a finite collection of
  hyperplanes, so there is a uniform bound on the distance from a
  point in one of these hyperplanes to a generic point.

  We set up some constants now for later use.  This is just to
  emphasize that they depend only on the subgroup we are given and the
  data we have been handed thus far.  Let $\mathcal{G}$ be the finite
  set of marked graphs provided by Proposition
  \ref{prop:finite-graphs} and define $K_2$ as the maximum of
  $BCC(\rho_G)$ and $BCC(\rho^{-1}_G)$ as $G$ varies over the finitely
  many marked graphs in $\mathcal{G}$.  Lemma \ref{lem:marking-twists}
  then implies that $\tw(\rho(w))\ge \tw(w)-K_2$ for any conjugacy
  class $w$ and any of the finitely many marked graphs in
  $\mathcal{G}$.  Let $D_2$ be the constant from Lemma
  \ref{lem:generator-twists}.

  Fix a minimal generating set $\phi_1,\ldots,\phi_k$ for $H$ and let
  $\psi=\phi_1^{p_1}\cdots\phi_k^{p_k}$ be generic in $H$.  Let
  $f\colon G\to G$ be a CT representing $\psi$ with $G$ chosen from
  $\mathcal{G}$ and let $\omega$ be the comparison homomorphism for
  which $\omega(\psi)$ is the largest.  The key point is that given
  $\psi$, Corollary \ref{cor:QE-circuit} will provide a split circuit
  $\sigma$ for which the twisting will grow by $|\omega(\psi)|$ under
  application of the map $f$.

  Indeed, let $\sigma$ be the circuit provided by Corollary
  \ref{cor:QE-circuit}.  As we discussed in section
  \ref{sec:comp-homos}, there is a correspondence between the
  comparison homomorphisms for $H$ and the set of linear edges and
  quasi-exceptional families in $G$.  Assume first that $\omega$
  corresponds to the linear edge $E$ with axis $u$, so that by
  definition $f(E)=E\cdot u^{\omega(\psi)}$.  Since the splitting of
  $f_\#(\sigma)$ refines that of $\sigma$ and $E$ is a term in the
  complete splitting of $\sigma$, $f_\#(\sigma)$ not only contains the
  path $E\cdot u^{\omega(\psi)}$, but in fact splits at the ends of
  this subpath.  Under iteration, we see that $f_\#^t(\sigma)$
  contains the path $E\cdot u^{t\omega(\psi)}$, and therefore
  $\tw(f_\#^t(\sigma))\geq t|\omega(\psi)|$.  This isn't quite good
  enough for our purposes, so we will argue further to conclude that
  for some $t_0$,
  \begin{equation}\label{eq:1}
    \tw(f_\#^{t_0}(\sigma))-\tw(f_\#^{t_0-1}(\sigma))\geq |\omega(\psi)|
  \end{equation}
  Suppose for a contradiction that no such $t$ exists.  Then for every
  $t$, we have
  $\tw(f^t_\#(\sigma))-\tw(f^{t-1}_\#(\sigma))\leq |\omega(\psi)|-1$.
  Using a telescoping sum and repeatedly applying this assumption, we
  obtain $\tw(f^t_\#(\sigma))-\tw(\sigma)\leq t|\omega(\psi)|-t$.
  Combining and rearranging inequalities, this implies
  \begin{equation*}
    \tw(\sigma)
    \geq \tw(f^t_\#(\sigma))+t-t|\omega(\psi)|
    \geq t|\omega(\psi)|+t-t|\omega(\psi)|=t
  \end{equation*}
  for all $t$, a contradiction.  This establishes the existence of
  $t_0$ satisfying equation \eqref{eq:1}.

  The above argument works without modification in the case that
  $\omega$ corresponds to a family of quasi-exceptional paths.  We now
  address the minor adjustment needed in the case that $\omega$
  corresponds to a family of exceptional paths,
  $E_i u^*\overline{E}_j$.  Let $f(E_i)=E_iu^{d_i}$ and
  $f(E_j)=E_ju^{d_j}$.  Since $\sigma$ contains both
  $E_i u^*\overline{E}_j$ and $E_j u^*\overline{E}_i$ in its complete
  splitting, we may assume without loss that $d_i>d_j$.  The only
  problem is that the exponent of $u$ in the term
  $E_i u^*\overline{E}_j$ occuring in the complete splitting of
  $\sigma$ may be negative, so that $\tw(f^t_\#(\sigma))$ may be less
  than $t|\omega(\psi)|$.  In this case, just replace $\sigma$ by a
  sufficiently high iterate so that the exponent is positive.

  Now write $\psi$ in terms of the generators $\psi=s_1s_2\cdots s_p$
  so that for any conjugacy class $w$, by repeatedly applying Lemma
  \ref{lem:generator-twists} we obtain
  \begin{equation*}
    \tw(s_1(s_2\cdots s_p(w)))
    \leq \tw(s_2(s_3\cdots s_p(w)))+D_2
    \leq \tw(s_3\cdots s_p(w))+2D_2
    \leq\ldots\leq \tw(w)+pD_2
  \end{equation*}
  so that, $D_2|\psi|_{\Out(F_n)}\geq \tw(\psi(w))-\tw(w)$.  Applying
  this inequality to the circuit $f^{t_0-1}_\#(\sigma)$ just
  constructed, and letting $w$ be the conjugacy class
  $\rho^{-1}(f^{t_0-1}_\#(\sigma))$, we have
  \begin{equation*}\label{eq:2}
    |\psi|_{\Out(F_n)}
    \geq \frac{1}{D_2}\left[\tw(\psi(w))-\tw(w)\right]
    \geq\frac{1}{D_2}\left[\tw(f^{t_0}_\#(\sigma))-\tw(f^{t_0-1}_\#(\sigma))\right]-\frac{2K_2}{D_2}
    \geq\frac{1}{D_2}|\omega(\psi)|-\frac{2K_2}{D_2}
  \end{equation*}
  The second inequality is justified by Lemma \ref{lem:marking-twists}
  and the third uses the property of $\sigma$ established in
  \eqref{eq:1} above.  Since $\omega$ was chosen to be largest
  coordinate of $\Omega_{comp}(\psi)$ and $\Omega_{comp}$ is
  injective, the proof is complete.
\end{proof}

\section{The Mixed Case}
\label{sec:mixed-case}
There are no additional difficulties with the mixed case since both
the distance function on $\CV_n$ and Alibegovi\'{c}'s twisting
function are well suited for dealing with outer automorphisms whose
growth is neither purely exponential nor purely polynomial.
Consequently, for an element $\psi$ of an abelian subgroup $H$, if the
image of $\psi$ is large under $PF_H$ then we can use $\CV_n$ to show
that $|\psi|_{\Out(F_n)}$ is large, and if the image is large under
$\Omega_{comp}$ then we can use the methods from \S \ref{sec:PG-case}
to show $|\psi|_{\Out(F_n)}$ is large.  The injectivity of $\Omega$
\cite[Lemma 4.6]{FH-AB} exactly says that if $|\psi|_H$ is large, then
at least one of the aforementioned quantities must be large as well.

\begin{thm}
  \label{thm:undistorted}
  Abelian subgroups of $\Out(F_n)$ are undistorted.
\end{thm}
\begin{proof}
  Assume, by passing to a finite index subgroup, that $H$ is
  rotationless.  By \cite[Lemma 4.6]{FH-AB}, the map
  $\Omega\colon H\to \mathbb{Z}^{N+K}$ is injective.  Choose a minimal
  generating set for $H$ and write
  $H=\langle \phi_1,\ldots,\phi_k\rangle$.  The restriction of
  $\Omega$ to the first $N$ coordinates is precisely the map $PF_H$
  from section \ref{sec:EG-case}.  Choose $k$ coordinates of $\Omega$
  so that the restriction $\Omega_\pi$ to those coordinates is
  injective.  Let $PF_{\Lambda_1},\ldots,PF_{\Lambda_l}$ be the subset
  of the chosen coordinates corresponding to expansion factor
  homomorphisms.  Pass to a finite index subgroup of $H$ and choose
  generators so that
  $\Omega_\pi(\phi_i)=(0,\ldots,PF_{\Lambda_i}(\phi_i),\ldots,0)$ for
  $1\leq i\leq l$.  Now we proceed as in the proofs of Theorems
  \ref{thm:EG-case} and \ref{thm:PG-case}.

  Fix a basepoint $*\in\CV_n$ and let
  $\psi=\phi_1^{p_1}\cdots\phi_k^{p_k}$ in $H$.  We may assume without
  loss that $\psi$ is generic in $H$ (again, it suffices to prove that
  generic elements are uniformly undistorted).  Replace the $\phi_i$'s
  by their inverses if necessary to ensure that all $p_i$'s are
  non-negative.  Then, for each of the first $l$ coordinates of
  $\Omega_\pi$, replace $\Lambda_i$ by its paired lamination if
  necessary (Lemma \ref{lem:paired-lamination}) to ensure that
  $PF_{\Lambda_i}(\psi)>0$.  Look at the coordinates of
  $\Omega_\pi(\psi)$ and pick out the one with the largest absolute
  value.  We first consider the case where the largest coordinate
  corresponds to an expansion factor homomorphism $PF_{\Lambda_j}$.
  We have already arranged that $PF_{\Lambda_j}(\psi)>0$.

  By Theorem \ref{taulength}, the translation distance of $\psi$ is the
  maximum of the Perron-Frobenius eigenvalues associated to the EG
  strata of a relative train track representative $f$ of $\psi$.
  Since $\psi$ is generic and the first $l$ coordinates of
  $\Omega_\pi$ are non-negative,
  $\{\Lambda_1,\ldots,\Lambda_l\} \subset \mathcal{L}(\psi)$.  Each
  $\Lambda_i$ is associated to an EG stratum of $f$.  For such a
  stratum, the logarithm of the PF eigenvalue is
  $PF_{\Lambda_i}(\psi)$.  Just as in the proof of Theorem
  \ref{thm:EG-case}, for each $1\leq i\leq l$, we have that
  $PF_{\Lambda_i}(\psi)=p_i PF_{\Lambda_i}(\phi_i)$.  So the
  translation distance of $\psi$ acting on Outer Space is
  \begin{equation*}
    \tau(\psi)\geq\max\{p_iPF_{\Lambda_i}(\phi_i)\mid 1\leq i\leq l\}    
  \end{equation*}
  The inequality is because there may be other laminations in
  $\mathcal{L}(\psi)$.  Just as in Theorem \ref{thm:EG-case}, we have
  \begin{equation*}
    d(*,*\cdot\psi)\leq D_1|\psi|_{\Out(F_n)}
  \end{equation*}
  where $D_1=\max_{s\in S}d(*,*\cdot s)$.  Let
  $K_1=\min\{PF_{\Lambda_i^\pm}(\phi_j^\pm)\mid 1\leq i\leq l, 1\leq
  j\leq k\}$.  Then we have
  \begin{equation*}
    |\psi|_{\Out(F_n)}
    \geq\frac{1}{D_1}d(*,*\cdot \psi)
    \geq\frac{1}{D_1}\tau(\psi)
    \geq\frac{1}{D_1}\max\{p_iPF_{\Lambda_i}(\phi_i)\mid 1\leq i\leq k\}
    \geq\frac{K_1}{D_1}\max\{p_i\}
  \end{equation*}

  We now handle the case where the largest coordinate of
  $\Omega_\pi(\psi)$ corresponds to a comparison homomorphism
  $\omega$.  Let $\mathcal{G}$ be the finite set of marked graphs
  provided by Proposition \ref{prop:finite-graphs} and let
  $f\colon G\to G$ be a CT for $\psi$ where $G\in\mathcal{G}$.  Define
  $K_2$ exactly as in the proof of Theorem \ref{thm:PG-case} so that
  $\tw(\rho(w))\geq \tw(w)-K_2$ for all conjugacy classes $w$ and any
  marking or inverse marking of the finitely many marked graphs in
  $\mathcal{G}$.  The construction of the completely split circuit
  $\sigma$ satisfying equation \eqref{eq:1} given in the polynomial
  case works without modification in our current setting, where the
  comparison homomorphism $\omega$ in equation \eqref{eq:1} is the
  coordinate of $\Omega_\pi$ which is largest in absolute value.

  Using this circuit and defining $w=\rho^{-1}(f^{t_0-1}_\#\sigma)$,
  the inequalities and their justifications in the proof of Theorem
  \ref{thm:PG-case} now apply verbatim to the present setting to
  conclude
  \begin{equation*}
    |\psi|_{\Out(F_n)}
    \geq\frac{1}{D_2}\max\{|\omega(\psi)|\mid \omega\in\Omega_\pi\}|-\frac{2K_2}{D_2}
  \end{equation*}
  We have thus shown that the image of $H$ under $\Omega_\pi$
  undistorted.  Since $\Omega_\pi$ is injective, it is a
  quasi-isometric embedding of $H$ into $\mathbb{Z}^k$, so the theorem
  is proved.
\end{proof}

We conclude by proving the rank conjecture for $\Out(F_n)$.  The
maximal rank of an abelian subgroup of $\Out(F_n)$ is $2n-3$, so
Theorem \ref{thm:undistorted} gives a lower bound for the geometric
rank of $\Out(F_n)$: $\rk\Out(F_n)\leq 2n-3$.  The other inequality
follows directly from the following result, whose proof we sketch
below.

\begin{thm}\label{thm:vcd-rank}
  If $G$ has virtual cohomological dimension $k\geq 3$, then $\rk(G)\leq k$.
\end{thm}

The virtual cohomological dimension of $\Out(F_n)$ is $2n-3$ \cite{CV-VCD}.  Thus, for $n\geq 3$, we have:

\begin{cor}
  \label{cor:rank}
  The geometric rank of $\Out(F_n)$ is $2n-3$, which is the maximal
  rank of an abelian subgroup of $\Out(F_n)$.
\end{cor}

\begin{proof}[Proof of \ref{thm:vcd-rank}]
  Let $G'\leq G$ be a finite index subgroup whose cohomological
  dimension is $k$.  Since $G$ is quasi-isometric to its finite index
  subgroups, we have $\rk(G')=\rk(G)$.  A well known theorem of
  Eilenberg-Ganea \cite{EG-GD} provides the existence of a
  $k-$dimensional CW complex $X$ which is a $K(G',1)$.  By
  \v{S}varc-Milnor, it suffices to show that there can be no
  quasi-isometric embedding of $\mathbb{R}^{k+1}$ into the universal
  cover $\tilde{X}$.  Suppose for a contradiction that
  $f\colon \mathbb{R}^{k+1}\to \tilde{X}$ is such a map.  The first
  step is to replace $f$ by a continuous quasi-isometry $f'$ which is
  a bounded distance from $f$.  This is done using the
  ``connect-the-dots argument'' whose proof is sketched in
  \cite{SW-QI}.  The key point is that $\tilde{X}$ is \emph{uniformly
    contractible}.  That is, for every $r$, there is an $s=s(r)$, such
  that any continuous map of a finite simplicial complex into $X$
  whose image is contained in an $r$-ball is contractible in an
  $s(r)$-ball.

  It is a standard fact \cite[Theorem 2C.5]{Hat-AT} that $X$ may be
  replaced with a simplicial complex of the same dimension so that
  $\tilde{X}$ may be assumed to be simplicial.  We now construct a
  cover $\mathcal{U}$ of the simplicial complex $\tilde{X}$ whose
  nerve is equal to the barycentric subdivision of $\tilde{X}$.  The
  cover $\mathcal{U}$ has one element for each cell of $\tilde{X}$.
  For each vertex $v$, the set $U_v\in\mathcal{U}$ is a small
  neighborhood of $v$.  For each $i$-cell, $\sigma$, Define $U_\sigma$
  by taking a sufficiently small neighborhood of
  $\sigma\setminus \bigcup_{\sigma'\in \tilde{X}^{(i-1)}}U_{\sigma'}$
  to ensure that $U_\sigma\cap \tilde{X}^{(i-1)}=\emptyset$.  The key
  property of $\mathcal{U}$ is that all $(k+2)-$fold intersections are
  necessarily empty because the dimension of the barycentric
  subdivision of $\tilde{X}$ is equal to $\dim(\tilde{X})$.

  Since we have arranged $f$ to be continuous, we can pull back the
  cover just constructed to obtain a cover
  $\mathcal{V}=\{f^{-1}(U)\}_{U\in\mathcal{U}}$ of $\mathbb{R}^{k+1}$.
  Since the elements of $\mathcal{U}$ are bounded, and $f$ is a
  quasi-isometric embedding, the elements of $\mathcal{V}$ are bounded
  as well.  The intersection pattern of the elements of $\mathcal{V}$
  is exactly the same as the intersection pattern of elements of
  $\mathcal{U}$.  But the cover $\mathcal{U}$ was constructed so that
  any intersection of $(k+2)$ elements is necessarily empty.  Thus, we
  have constructed a cover of $\mathbb{R}^{k+1}$ by bounded sets with
  no $(k+2)-$fold intersections.  We will contradict the fact that the
  Lebesgue covering dimension of any compact subset of
  $\mathbb{R}^{k+1}$ is $k+1$.  Let $K$ be compact in
  $\mathbb{R}^{k+1}$ and let $\mathcal{V}'$ be an arbitrary cover of
  $K$.  Let $\delta$ be the constant provided by the Lebesgue covering
  Lemma applied to $\mathcal{V}'$.  Since the elements of
  $\mathcal{V}$ are uniformly bounded, we can scale them by a single
  constant to obtain a cover of $K$ whose sets have diameter
  $<\delta/3$.  Such a cover is necessarily a refinement of
  $\mathcal{V}'$, but has multiplicity $k+1$.  This contradicts the
  fact that $K$ has covering dimension $k+1$ so the theorem is proved.
\end{proof}
\bibliographystyle{alpha}%
\bibliography{bibliography}
\end{document}